\documentclass[11pt,a4paper]{article}
\usepackage[utf8]{inputenc}
\usepackage{amssymb}
\usepackage{graphicx}
\usepackage{amsmath,amsfonts}
\usepackage{color}
\usepackage{fancyhdr}
\def\qed{\hfill$\Box$}  

\def\R{\mathbb{R}}  
\def\N{\mathbb{N}}

\def\R{\mathbb{R}}  
\def\N{\mathbb{N}}  
\def\Z{\mathbb{Z}}

\def\loc{\mathrm {loc}}
\def\R{\mathbb{R}}
\def\N{\mathbb{N}}
\def\calH{\mathcal{H}}

\providecommand{\supp}{\operatorname{supp}}
\providecommand{\dr}{\, \mathrm{d} r}
\providecommand{\dt}{\, \mathrm{d} t}
\providecommand{\ds}{\, \mathrm{d} s}
\providecommand{\dx}{\, \mathrm{d} x}
\providecommand{\dy}{\, \mathrm{d} x_2}
\providecommand{\dz}{\, \mathrm{d} z}
\providecommand{\dxi}{\, \mathrm{d} x_\xi}
\providecommand{\deta}{\, \mathrm{d} x_\eta}

\providecommand{\com}{\color{black}}
\providecommand{\cm}{\color{black}}
\providecommand{\BZ}{\color{black}}
\providecommand{\SC}{\color{black}}
\definecolor {red} {rgb} {0.3, 0.0, 0.0}

\newcommand{\dist}{\operatorname{dist}}
\newcommand{\curl}{\operatorname{curl}}
  
\def\eps{\varepsilon}  
\newtheorem{theorem}{Theorem}[section]  
\newtheorem{lemma}[theorem]{Lemma}

\newtheorem{corollary}[theorem]{Corollary}  
\newtheorem{proposition}[theorem]{Proposition}
  
\newtheorem{remark}[theorem]{Remark}

\newenvironment{proof}{{\noindent\bf 
Proof:}}{\hfill\qed \bigskip \newline } 

\numberwithin{equation}{section}
\title{Low volume-fraction microstructures\\
in martensites and crystal plasticity}
\author{Sergio Conti and Barbara Zwicknagl\\
\emph{\small Institut für Angewandte Mathematik, Universität  Bonn,
53115 Bonn}}
\date{July 16, 2015}

\begin{document}
\maketitle
\begin{abstract}
We study microstructure formation in two nonconvex singularly-perturbed variational problems
from materials science, one modeling austenite-martensite interfaces
in shape-memory alloys, the
other one slip structures in the plastic deformation of crystals.
For both functionals we determine the scaling of the
optimal energy in terms of the parameters of the problem, leading to a characterization of the 
mesoscopic phase diagram. Our results identify the presence of a new phase,
which is intermediate between the classical laminar microstructures and branching
patterns. The new phase, characterized by partial branching, appears for both problems 
in the limit of small volume fraction,
that is, if one of the variants (or of the slip systems) dominates the picture and the volume
fraction of the other one is small. 
\end{abstract}

\section{Introduction}
The study of spontaneous pattern formation in materials constitutes an important application of the calculus
of variations to materials science. From a variational viewpoint, the origin of microstructure is related
to a nonconvexity of the energy density, and to boundary conditions which favour states which correspond to 
a mixture of  different
minima of the energy density. 
In continuum mechanics typically  the independent variable is the gradient of a vector field, which 
obeys the zero-curl differential condition, leading to strong constraints on the admissible microstructures.
The theory of relaxation studies the effective behavior of nonconvex variational problems
which lack lower semicontinuity and possibly existence of  minimizers, but does not give
detailed information on the type of microstructure expected \cite{ball-james:87,MuellerLectureNotes,Dacorognabuch}.

A finer analysis can be done if a regularization is included, in the form of a convex higher-order term, 
which physically may represent interfacial energies. An exact determination
of the minimizers and the minimal energy is, for these more complex problems, typically impossible.
Already a study of the optimal scaling of the energy with respect to the parameters of the problem
may, however, give very valuable information. Starting with the works of Landau
\cite{landau:38,landau:43} on micromagnetism, branching-type patterns have been predicted and observed.
They are characterized by coarse oscillations in the interior, which refine close to the boundary,
as illustrated in  Figure \ref{figlambr}. At
a heuristic level, the transition between coarse and fine oscillations can be understood as the
result of the competition between the miminization of the  total length of the interfaces,
the energetic cost of bending the domains, and the boundary conditions.

The mathematical study of  the subject began with the work of Kohn and M\"uller in the 90s \cite{kohn-mueller:92,kohn-mueller:94}, who proposed a simple
scalar model for martensitic microstructures close to an interface with austenite, see Section~\ref{sec:introkm} below for details. 
Their basic finding was the presence of a transition between a regime in which the energy scales proportional to 
$\eps^{1/2}\mu^{1/2}$, with $\eps$ being the surface energy density and $\mu$ the ratio between
the austenite and the martensite elastic coefficients, and a regime in which the energy scales proportional to $\eps^{2/3}$. 
The first regime corresponds to a one-dimensional laminar pattern, the second one to a branching-type pattern, as illustrated 
in Figure \ref{figlambr}.

\begin{figure}
\centerline{ \includegraphics[height=5cm]{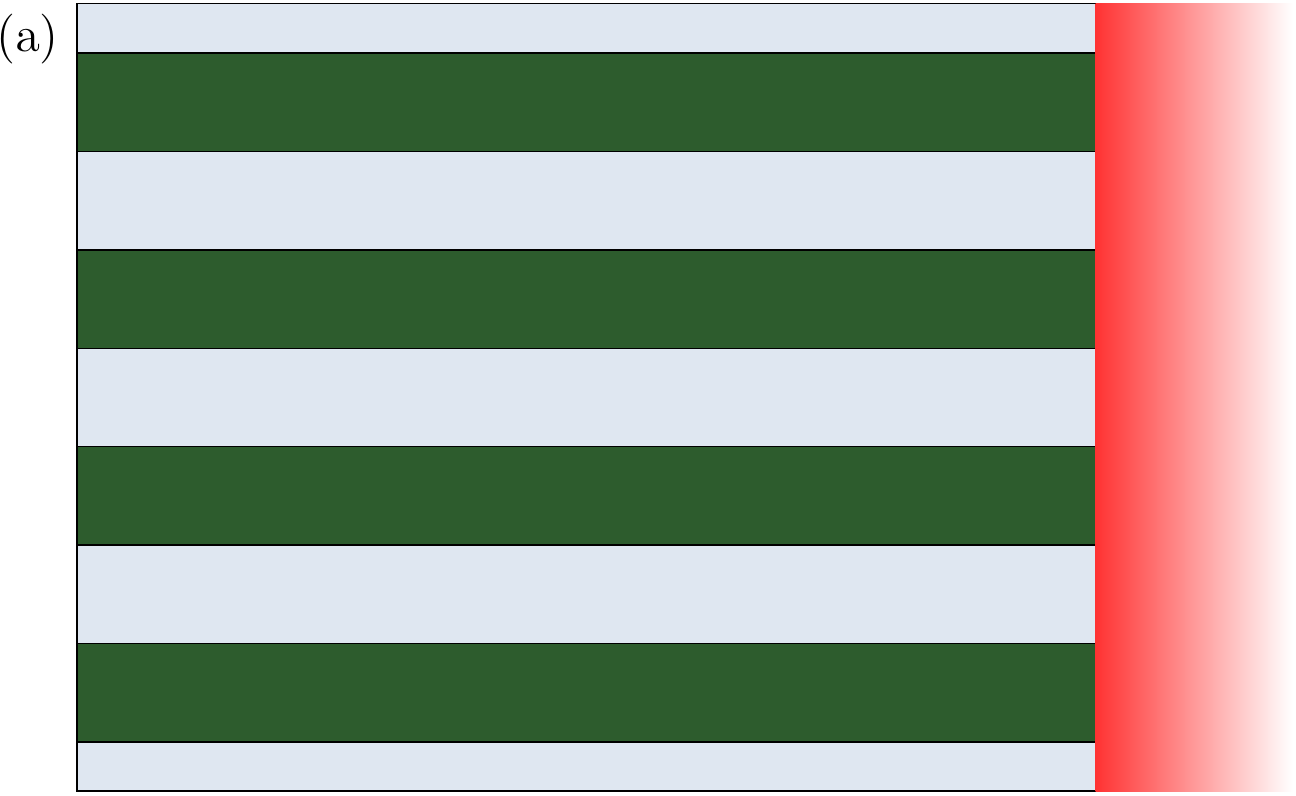}\hskip.8cm
 \includegraphics[height=5cm]{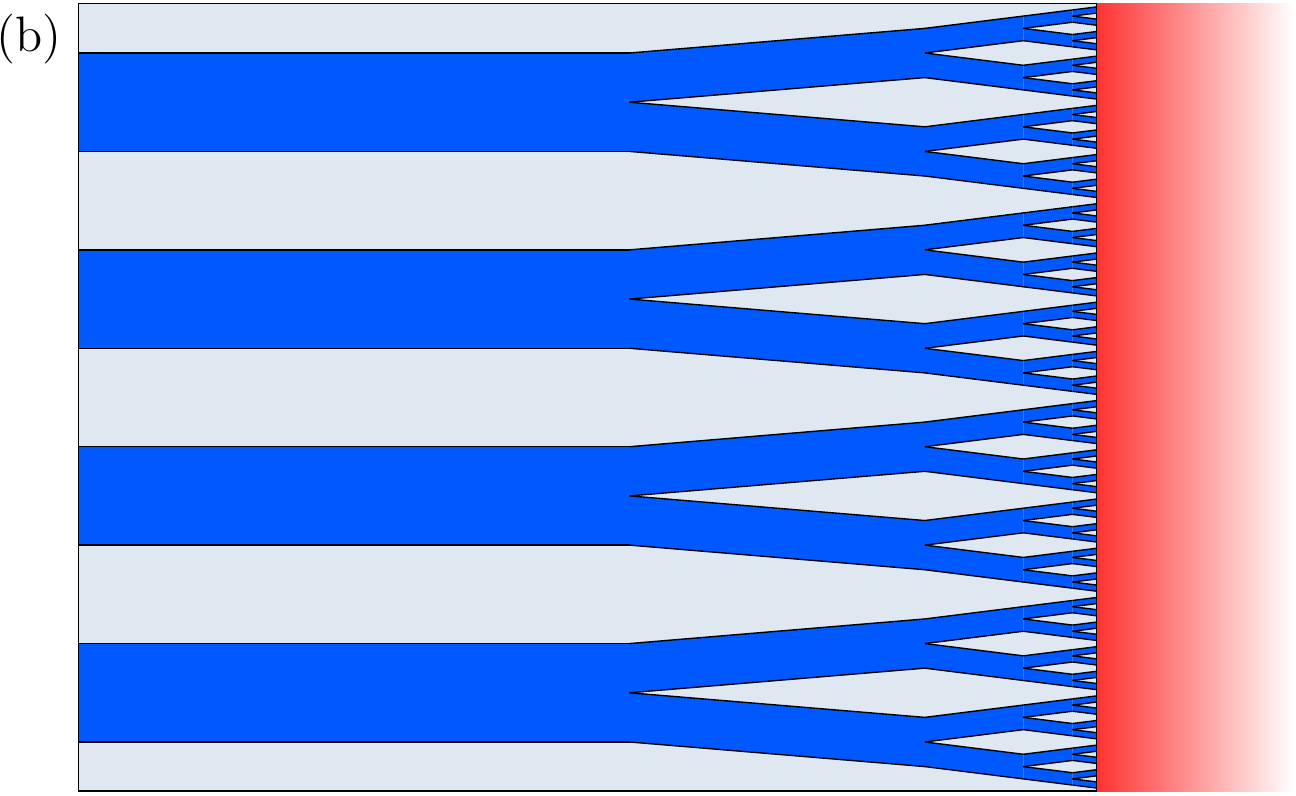}}
 \caption{Left: sketch of a laminate construction. Right: sketch of the branching construction,
 both for volume fraction 1/2 as in the original Kohn-Müller model. In both cases the interface to austenite
 is on the right.}
 \label{figlambr}
\end{figure}

Their results have been  refined in specific regimes  \cite{conti:00,giuliani-mueller:12},
  extended to related scalar-valued models in different parameter regimes \cite{schreiber:94,zwicknagl:14}
and to vector-valued models  \cite{capella-otto:09,knuepfer-kohn:11,CapellaOtto2012,knuepfer-kohn-otto:13,diermeier:13,chan-conti:14,chan-conti:14-1,bella-goldman:15}.
Similar results have been obtained in other variational models, including the 
magnetic structures in ferromagnets originally studied by Landau \cite{choksi-kohn:98,choksi-et-al:98,otto-viehmann:10,knuepfer-muratov:11}, 
flux-domain patterns in type-I superconductors  \cite{choksi-et-al:08,conti-et-al:15}, 
diblock copolymers \cite{Choksi01,alberti-et-al:09}, 
blistering of thin compressed films  \cite{BCDM00,JinSternberg2,JinSternberg1,belgacem-et-al:02}, wrinkling of stretched thin films \cite{bella-kohn:14}, dislocation patterns in crystal plasticity \cite{conti-ortiz:05}, and compliance minimization \cite{Kohn-Wirth:14-1,Kohn-Wirth:15}. 

In this paper we specifically focus on two problems in this class where the expected microstructure is essentially two-dimensional. 
The first one, and the simpler one, is the Kohn-Müller functional.  The second one is the scalar, but three-dimensional, model of crystal plasticity
from \cite{conti-ortiz:05}.  We present the first model, its physical interpretation and the main results on this functional
in Section \ref{sec:introkm}, the other one in Section \ref{sec:introplastic}. After mentioning some notation in 
Section \ref{secnotation}, we give the proofs for the scaling laws for {\com{the}} two models in Section{\com{s}} \ref{secpfKM} and \ref{sec:disloc} respectively.

\section{The model for martensitic microstructures}\label{sec:introkm}
\begin{figure}
\centerline{
 \includegraphics[width=12cm]{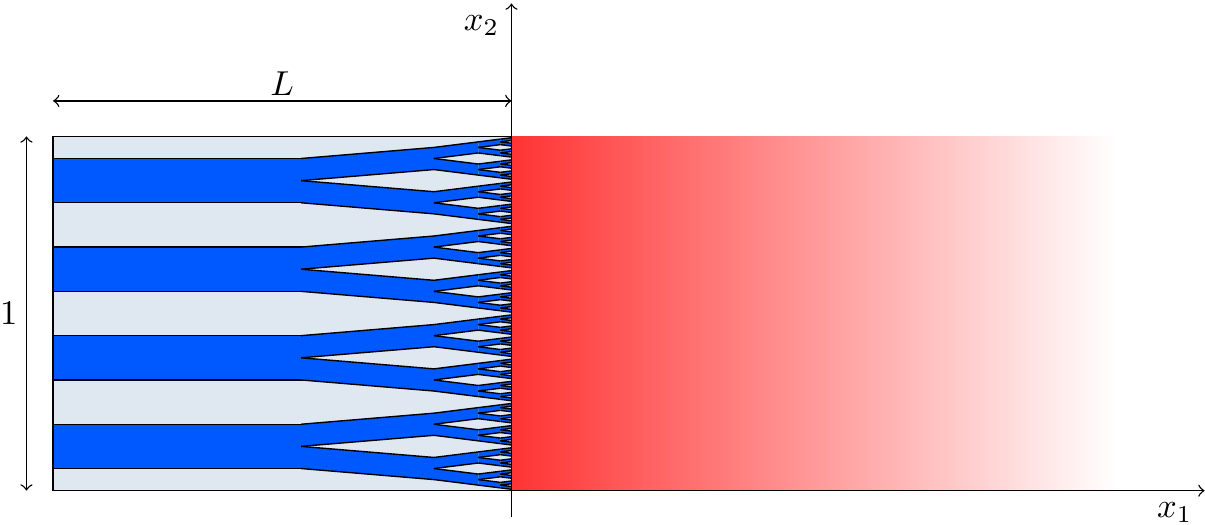}}
 \caption{Sketch of the geometry in the Kohn-Müller model. The shaded area on the right represents the austenite, which 
 extends to infinity.}
 \label{fig-geometry}
\end{figure}
We now introduce our model for martensitic microstructures. Following Kohn and Müller \cite{kohn-mueller:92,kohn-mueller:94}
we work for simplicity in two dimensions, in antiplane shear geometry, with the scalar field $u$ 
representing a deformation in the out-of-plane direction.
We consider one interface between austenite and martensite, located at $\{0\}\times (0,1)$, 
with austenite on the right and martensite on the left, as sketched in Figure \ref{fig-geometry}.
In the austenite, which for simplicity is assumed to extend to infinity, the minimum of the elastic energy is attained at $\nabla u=0$.
The coefficient $\mu>0$ represents the ratio of the elastic coefficients of austenite and martensite, respectively
(this parameter is often denoted by $\beta$ in the literature).
In the martensite, which we assume to cover the domain $(-L,0)\times (0,1)$, there are two minima of the elastic energy density.
After scaling, we can write the energy as 
\begin{equation}\label{eq:KM}
  J(u):=\mu\int_0^\infty\int_0^1 |\nabla u(x)|^2\dx+\int_{-L}^0\int_0^1 (\partial_1u(x))^2 \dx+ \eps\int_{-L}^0\int_0^1 |\partial_2\partial_2u| 
\end{equation}
where admissible functions $u\in W^{1,2}_\loc((-L,\infty)\times(0,1))$ satisfy $\partial_2u\in \{\theta
,-1+\theta \}$ almost everywhere in $(-L,0)\times(0,1)$, for some $\theta\in (0,1/2{\com{]}}$. 
We denote by $\partial_iu:=\frac{\partial}{\partial x_i}u$ the distributional derivative, in the last
term $\partial_2\partial_2 u$ is assumed to be a measure and the term has to be understood distributionally. The 
partial derivative $\partial_2u$ represents the order parameter of the martensitic phase transition, and its preferred values are dictated by crystallography.
The first two terms in \eqref{eq:KM} model the elastic energy contributions of the austenite and the  martensite part, respectively. The last term in \eqref{eq:KM} is a regularization term which penalizes changes in the order parameter, and prevents arbitrarily fine microstructures. It can thus be interpreted as a surface energy term, $\eps$ being a typical surface energy constant per unit length. The energy functional is normalized to set the elastic modulus of martensite to one. \\

Mathematically, the parameter $\theta$ represents a compatibility condition. For $\theta=0$, there exist trivial 
configurations with vanishing energy, while for $\theta>0$, the minimal energy is strictly positive, 
and we expect the formation of microstructures. By swapping the two variants of microstructure one
can assume without loss of generality that  $0<\theta\leq 1/2$. 
Experimental findings suggest that the size of $\theta$ is closely linked to the width of the thermal hysteresis  loop
in the austenite-martensite phase transition {\BZ{\cite{james-zhang:05,cui-et-al:06,zhang:07,zarnetta-et-al,louie-et-al,Sriva-et-al,bechthold-et-el:12}, and such low-hysteresis alloys have been found to exhibit peculiar microstructures, see e.g. \cite{delville-et-al2,delville-et-al,Shi-et-al:14}}}.

\begin{figure}
\centerline{
 \includegraphics[width=7cm]{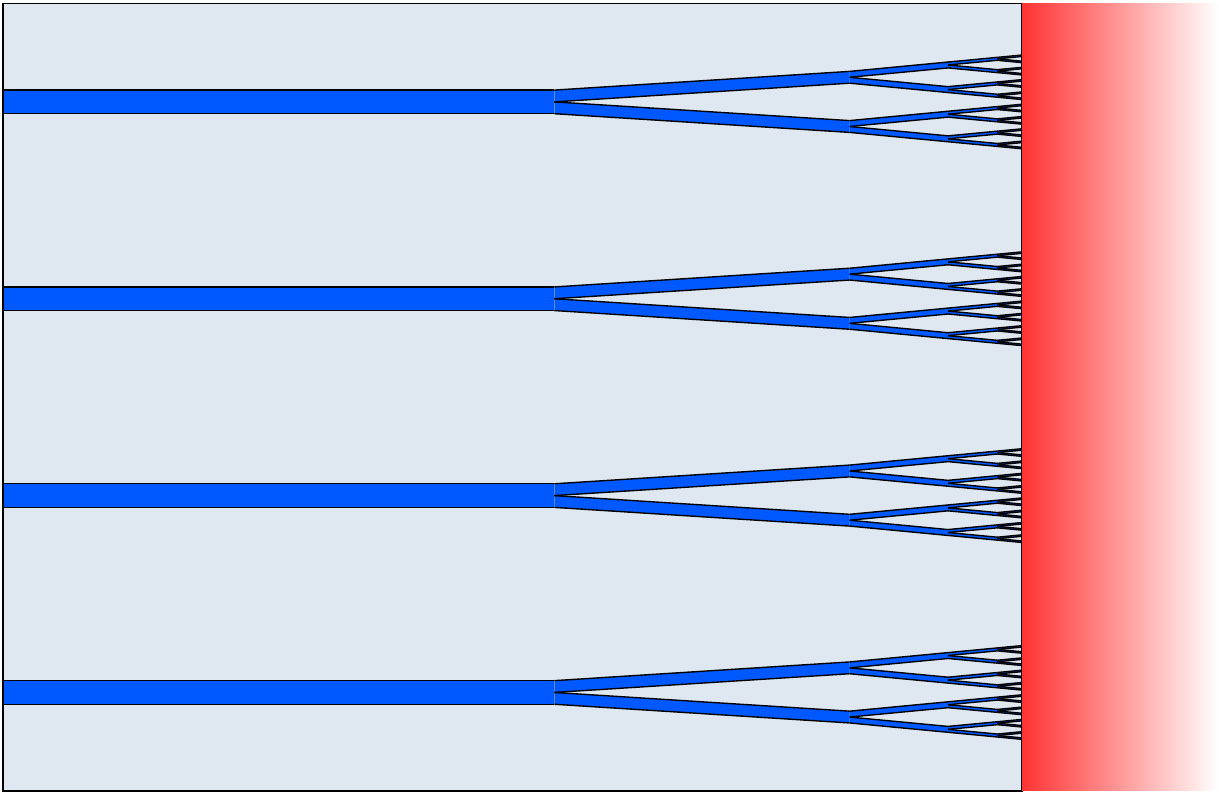}\hskip1cm
 \includegraphics[width=7cm]{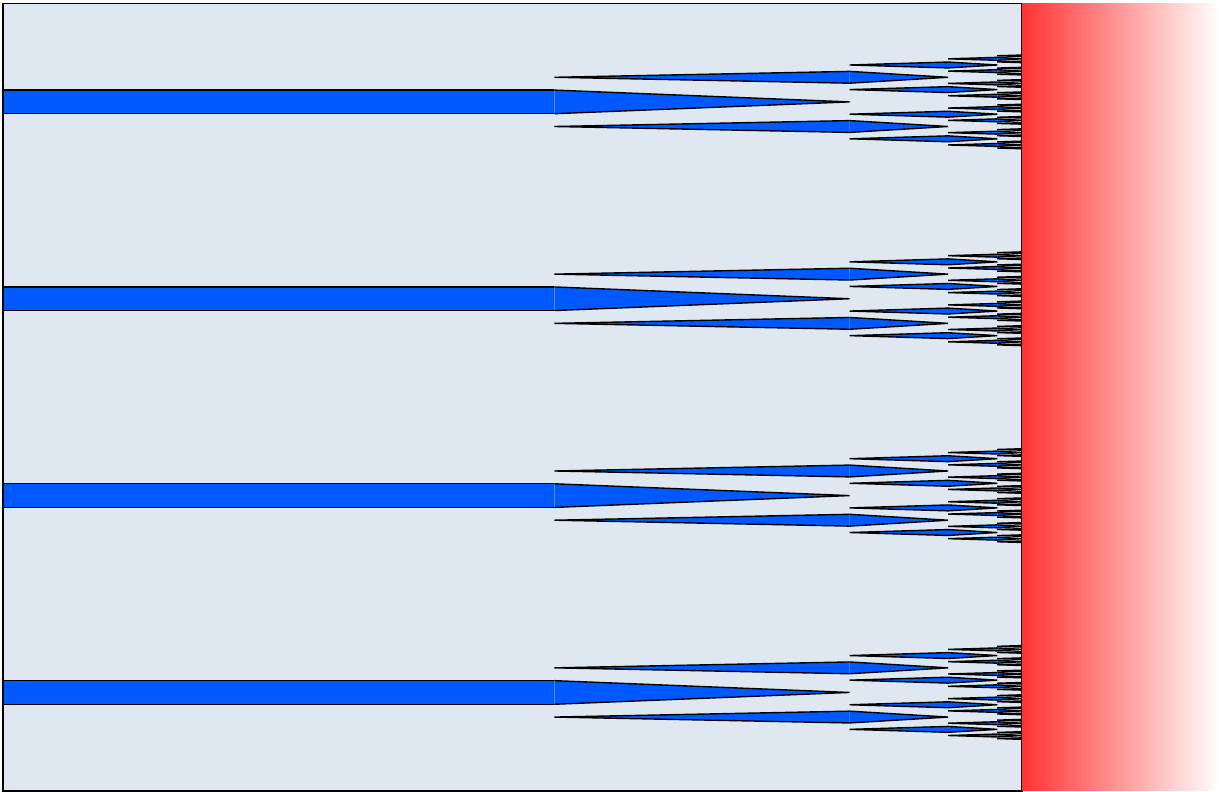}}
\caption{Sketch of the new two-scale-branching regime.
Left: geometry in which the minority phase is connected. Right: geometry in which the majority phase is connected.}
\label{fig:fig-new}
\end{figure}

For the symmetric case $\theta=\frac{1}{2}$, the scaling law for the minimal energy \eqref{eq:KM}  has been 
studied in \cite{kohn-mueller:92,kohn-mueller:94,conti:06}. 
We point out that the same scaling law for \eqref{eq:KM} holds if one restricts to periodic functions $u$, i.e., 
$u(x,0)=u(x,1)$, and  if $\mu\geq 1$, then the scaling results for $\theta
=\frac{1}{2}$ can be generalized to the case $\theta
<\frac{1}{2}$ with the obvious modifications in the lower bound, and the same kinds of test
functions to prove the upper bound (see \cite{zwicknagl:14}). In all cases, one observes only scaling regimes 
that correspond to uniform structures, to laminated structures or to geometrically refining branching patterns (see \cite{kohn-mueller:94,zwicknagl:14,diermeier:10}). We shall show in this paper that for $\mu\ll 1$ there is an intermediate regime between laminates and branching,
with the geometry sketched in Figure~\ref{fig:fig-new}.
A related intermediate regime between the branching constructions and the laminates in case of highly 
unequal volume fractions has been observed in a three-dimensional model of type-I superconductors  \cite{choksi-et-al:98,choksi-et-al:08,conti-et-al:15}. 
{\com
The scaling of the energy was different, and also the geometry was not the same as here. 
Indeed, in the superconducting case 
the conservation of flux forces the minority (normal) phase to be connected, as in the left panel of Figure \ref{fig:fig-new}. 
Here a second construction is possible, in which the majority phase is connected, as in the right panel of Figure \ref{fig:fig-new},
which has a smaller surface energy, as will become clear in Lemma \ref{lem:branch} below.
For the dislocation problem only the second construction gives the optimal energy scaling.}

%It turns out that 
%\begin{eqnarray}\label{eq:scalingConti}
%\min J(u)\sim\{\eps^{2/3}L^{1/3},\,\eps^{1/2},\,(\eps L\mu)^{1/2},\,\mu\}.
%\end{eqnarray}
%The first two regimes are achieved by geometrically refining branching patterns (see \cite{kohn-mueller:94,zwicknagl:14,diermeier:10}). Depending on the length $L$ of the transition layer, one expects microstructures on the whole domain or only in the vicinity of the interface. The third regime in \eqref{eq:scalingConti} is realized by laminates, and the last regime corresponds to a ``simple'' laminate, i.e., an affine function $u$. \\
We consider the case of general $\mu$, in particular $\mu\ll 1$. It turns out that in this case the choice 
of boundary conditions at the top and bottom boundaries matter. Precisely, we consider two natural classes of admissible functions,
the first one with Neumann boundary conditions on the horizontal sides,
\begin{eqnarray}
 \label{eq:A1}
\mathcal{A}_N&:=&\{u\in W^{1,2}_{\com{\text{loc}}}((-L,\infty)\times(0,1)):\ \partial_2u\in\{\theta, 1-\theta\} \text{\ a.e. in $(-L,0)\times(0,1)$ and} \nonumber\\
&& \partial_2\partial_2u\text{\ finite signed Radon measure} \},
\end{eqnarray}
and the one with periodic boundary conditions on the horizontal sides,
\begin{eqnarray}
 \label{eq:A2}
\mathcal{A}_P:=\{u\in\mathcal{A}_N:\ u(\cdot,0)=u(\cdot,1)\}.
\end{eqnarray}
%Note that the last term in \eqref{eq:KM} has to be understood in a distributional sense.
We prove the following scaling laws for $J$ (see Propositions \ref{prop:KMub} and \ref{prop:KMlb}). 
\begin{theorem} \label{th:1}
For all $\eps$, $\mu$, $L>0$ and all $0<\theta\leq\frac{1}{2}$, we have with $\hat{\eps}:=\eps/\theta^2$,
\begin{eqnarray}\label{eq:scaling1}
\min_{u\in\mathcal{A}_N} J(u)\sim \theta^2\min\left\{\hat{\eps}^{2/3}L^{1/3},\,\hat{\eps}^{1/2},\,(\hat{\eps} L\mu)^{1/2}\ln^{1/2}\left(3+\frac{\hat{\eps}}{\mu^3 L}\right),\,\left(\hat{\eps} L\mu\right)^{1/2}\ln^{1/2}\left(\frac{1}{\theta}\right),\, \mu\right\},\nonumber\\
\quad
\end{eqnarray}
and
\begin{eqnarray}\label{eq:scaling2}
\min_{u\in\mathcal{A}_P} J(u)\sim \theta^2\max\left\{\hat{\eps} L,\ \min\left\{\hat{\eps}^{2/3}L^{1/3},\ (\hat{\eps} L\mu)^{1/2}\ln^{1/2}\left(3+\frac{\hat{\eps}}{\mu^3 L}\right),\,\left(\hat{\eps} L\mu\right)^{1/2}\ln^{1/2}\left(\frac{1}{\theta}\right)\right\}\right\}.\nonumber\\
\quad
\end{eqnarray}
\end{theorem}
Here and in the rest of the paper $a\sim b$ means that there is a universal constant $c>0$ such that $\frac1c a \le b \le c\, a$.
The resulting phase diagrams are, for $L=1$, illustrated in Figure \ref{fig:fig2}.
\begin{remark}\label{rem:pvsn}
If $\theta\ll 1$, the choice of boundary conditions affects the scaling behavior in a non-trivial way. Precisely, for periodic boundary conditions we always have, by the constraint on $\partial_2 u$, the lower bound $\min_{u\in\mathcal{A}_P}J(u)\geq\eps L$. 
Then, for $\theta\approx \frac{1}{2}$, by Theorem \ref{th:1},
\begin{eqnarray*}
\min_{u\in\mathcal{A}_P} J(u)&\sim&\max\left\{\eps L,\ \min\left\{\eps^{2/3}L^{1/3},\ (\eps L\mu)^{1/2}\right\}\right\}=\max\left\{\eps L,\ \min\left\{\eps^{2/3}L^{1/3},\ (\eps L\mu)^{1/2},\ \mu\right\}\right\}\\
&\sim& \max\{\eps L,\ \min_{u\in\mathcal{A}_N}J(u)\},
\end{eqnarray*}
while for $\theta\ll\frac{1}{2}$, we have the different behavior
\begin{eqnarray*}
&&\min_{u\in\mathcal{A}_P} J(u)\sim \theta^2\max\left\{\hat{\eps} L,\ \min\left\{\hat{\eps}^{2/3}L^{1/3},\ (\hat{\eps} L\mu)^{1/2}\ln^{1/2}\left(3+\frac{\hat{\eps}}{\mu^3 L}\right),\,\left(\hat{\eps} L\mu\right)^{1/2}\ln^{1/2}\left(\frac{1}{\theta}\right)\right\}\right\}\\
&=& \theta^2\max\left\{\hat{\eps} L,\ \min\left\{\hat{\eps}^{2/3}L^{1/3},\ (\hat{\eps} L\mu)^{1/2}\ln^{1/2}\left(3+\frac{\hat{\eps}}{\mu^3 L}\right),\,\left(\hat{\eps} L\mu\right)^{1/2}\ln^{1/2}\left(\frac{1}{\theta}\right),\ \mu\ln\frac{1}{\theta}\right\}\right\}\\
&\not\sim& \theta^2\max\left\{\hat{\eps} L,\ \min\left\{\hat{\eps}^{2/3}L^{1/3},\ (\hat{\eps} L\mu)^{1/2}\ln^{1/2}\left(3+\frac{\hat{\eps}}{\mu^3 L}\right),\,\left(\hat{\eps} L\mu\right)^{1/2}\ln^{1/2}\left(\frac{1}{\theta}\right),\ \mu\right\}\right\}\\
&\sim&\max\{\eps L,\min_{u\in\mathcal{A}_N}J(u)\}.
\end{eqnarray*}
The origin of the logarithmic correction in case of periodic boundary conditions will be outlined in Remark \ref{rem:singlelam}.
\end{remark}
\begin{figure}
\centerline{
 \includegraphics[height=6cm]{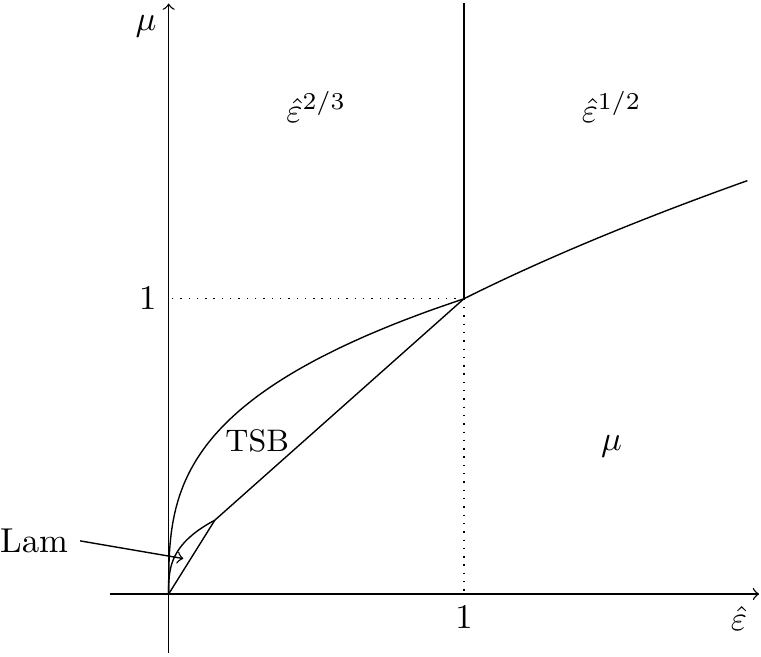}\hskip10mm
 \includegraphics[height=6cm]{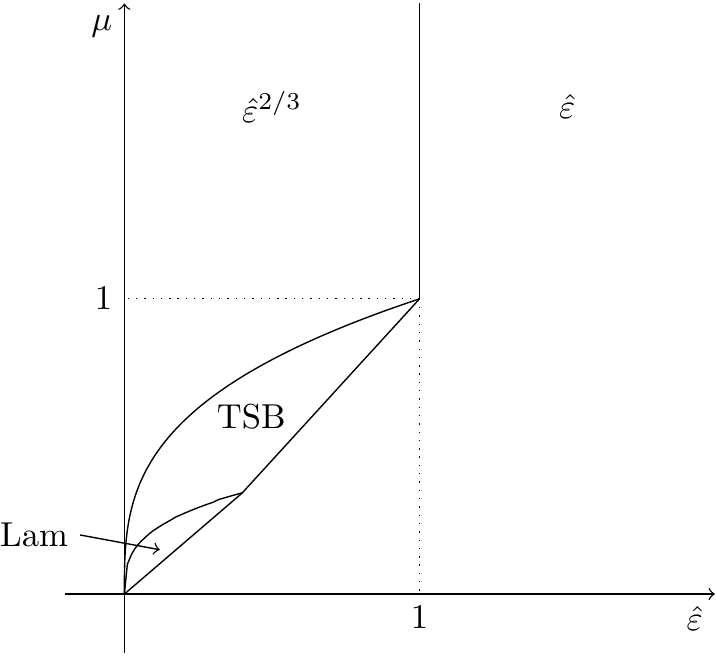}}
\caption{Phase diagrams as given in \eqref{eq:scaling1} (left) and \eqref{eq:scaling2} (right) with $L=1$.
The regions with energy proportional to $\hat\eps^{2/3}$ corresponds to the branching patterns of Figure 
\ref{figlambr}(b), the region marked ``TSB'' is where the new two-scale-branching regime illustrated in Figure \ref{fig:fig-new} is optimal,
the region marked ``Lam'' (very small in the left panel) is where the laminate pattern illustrated in 
\ref{figlambr}(a) is optimal.}
\label{fig:fig2}
\end{figure}

\section{The model for plastic microstructure}\label{sec:dislocmodel}
\label{sec:introplastic}
The second functional we study is a model for microstructures in crystal plasticity 
with large latent hardening that was proposed in \cite{conti-ortiz:05}.
Let  $\Omega:=(0,L)^3$ describe a grain of a plastic crystal, 
in which two slip systems are active. In a similar spirit as in the Kohn-Müller model 
of Section~\ref{sec:introkm} we focus on a scalar $u$, which represents a component of the deformation, and assume that 
it is affected by two slip systems, with slip plane normals
$e_\xi:=(e_1+e_2)/\sqrt2$ and $e_\eta:=(e_1-e_2)/\sqrt2$.
In the limit of large latent hardening the two slip systems cannot be both activated at the same point in 
space, therefore the plastic strain $\beta$ is pointwise parallel to either $e_\xi$ or $e_\eta$. 
We further assume that each slip system is only active with one orientation.
The class of admissible deformations and plastic strains therefore takes the form
\begin{eqnarray}\label{eq:disadm}
 \mathcal{A}:=\{(u,\beta)\in L^1(\Omega)\times L^1(\Omega;\R^3):\ 
 \beta_3=0,\ \beta_\xi\beta_\eta=0\text{\quad and\quad }\beta_\eta, \beta_\xi\geq 0\text{\quad a.e.}\}\,.
\end{eqnarray}
Here and in the rest of the paper, given a vector $v\in \R^3$ we denote the components along $\xi$ and $\eta$ by 
$v_\xi:=v\cdot e_\xi$ and $v_\eta:=v\cdot e_\eta$, and the Cartesian ones by $v_i:=v\cdot e_i$, $i=1,2,3$.

The energy is given by the sum of the elastic energy and the line-tension energy of the geometrically necessary
dislocations. Dislocations are topological defects in crystals which are responsible for plastic deformation 
processes, they are necessarily present whenever the plastic part of the deformation gradient is not a gradient field itself.
Geometrically necessary dislocations are the dislocations whose presence can be inferred from 
 the fact that the plastic strain $\beta$ is
not a gradient field. Starting from a variety of microscopic models it has been proven that the energetic contribution
of such dislocations is linear in the curl of $\beta$, with a coefficient which depends on the local Burgers
vector and slip-plane normal, see for example 
\cite{GarroniMueller2006,GarroniLeoniPonsiglione2010,ContiGarroniMueller2011,MuellerScardiaZeppieri2014,ContiGarroniOrtiz}. 

Working in the coordinates $(e_\xi,e_\eta,e_3)$ it is easy to see that the relevant terms
are $\partial_\xi\beta_\eta$ and $\partial_\eta\beta_\xi$. A detailed analysis of the latent-hardening condition 
from a relaxation viewpoint leads to the interpretation of $|\curl\beta|$ as $|\partial_\xi\beta_\eta|+|\partial_\eta\beta_\xi|$, where
both derivatives are interpreted distributionally, see  \cite{conti-ortiz:05} for a motivation.
In the present setting there is only one Burgers vector, since $u$ is a scalar. In order to study  the scaling
of the energy we can neglect the depencence on orientation and consider an isotropic penalization of $\curl\beta$. It is however
important that not all components of the gradient $D\beta$ enter the energy.

The elastic energy can be computed from the elastic strain, which is $Du-\beta$ in $\Omega$ and  $Du$ outside.
Denoting by $\eps>0$ the line-tension energy of a dislocation, the functional takes the form
{\com\begin{eqnarray}\label{eq:disloc}
 E(u,\beta):=\int_\Omega \left[|Du(x)-\beta(x)|^2 \dx+ \eps \left(|\partial_\xi\beta_\eta|+|\partial_3\beta_\eta|+ |\partial_\eta\beta_\xi|
 + |\partial_3\beta_\xi|\right)\right]
 +\mu\int_{\R^3\setminus\Omega} |D(u-u_0)(x)|^2\dx\,.\nonumber\\
 \qquad
\end{eqnarray}
%The exterior contribution, just like in the case of martensites, can be equivalently written in terms of the homogeneous $H^{1/2}$ norm. 
As in (\ref{eq:KM}), the terms $\partial_i\beta_j$ are interpreted distributionally.
}The relative activity of the two slip systems is fixed by the far-field deformation $u_0$, which will be taken affine.
We refer to \cite{conti-ortiz:05} for a more detailed physical motivation of the model and to
 \cite{AnD14,DondlMicroplast} for a discussion of the relaxation of $E$.

The energy scaling for the case of the far-field deformation $u_0(x)=x_1$, corresponding to a half-half mixture of the two slip directions $e_\xi$ and $e_\eta$, was studied in  \cite{conti-ortiz:05}. 
The key result was that the optimal energy scales, for small $\eps$, as the minimum of 
$\eps^{1/2}\mu^{1/2}$ and $\eps^{2/3}$, similarly to the Kohn-M\"uller model. 
The first scaling regime corresponds to the experimentally known Hall-Petch law, which states that
the critical stress for plastic deformation in a polycrystal scales as the grain size to the power $-1/2$.

Here we consider a situation in which 
the boundary data favor states in which one of the two slip systems dominates, corresponding to a macroscopic
forcing of the type
\begin{eqnarray}\label{eq:u0}
 u_0(x):=(1-\theta)x_\xi+\theta x_\eta
\end{eqnarray}
where $\theta\in (0,1/2)$. 
 We  derive the following scaling law for $E$ as defined in \eqref{eq:disloc} (see Propositions \ref{prop:dislocub} and  \ref{prop:disloclb}).
\begin{theorem}\label{th:2}
For all $\eps$, $L$, $\mu>0$ and all $\theta\leq\frac{1}{2}$, we have with $\hat{\eps}:=\eps/(L\theta^2)$,
 \[\inf_\mathcal{A} E(u,\beta)\sim L^3 \theta^2\min\left\{\hat{\eps}^{2/3},\,(\hat{\eps} \mu)^{1/2}\ln^{1/2}\left(3+\frac{\hat{\eps}}{\mu^3 }\right),\,\left(\hat{\eps} \mu\right)^{1/2}\ln^{1/2}\left(\frac{1}{\theta}\right),\, \mu,\, 1\right\}.\]
\end{theorem}
Similarly to the martensite case, we recover all regimes from the case of equal volume fraction, and show that there is an additional intermediate scaling regime, which is achieved by a two-scale branching construction. The proof of this result is 
given in Section~\ref{sec:disloc} below.
\section{Preliminaries and notation}\label{secnotation}
Throughout the text, we denote by $c$ positive constants that do not depend on any of the parameters $\eps$, $L$, $\theta$ or $\mu$ but that may change from expression to expression. For $A$, $B>0$, we use the notation $A \lesssim B$ if there is $c>0$ such that $A\leq c B$, and similarly for $\gtrsim$ and $\sim$. For a measurable set $M\subset\R^n$, we denote by $|M|$ its $n$-dimensional Lebesgue measure. 
{\com For $f\in L^1(M)$, $\int_M|\partial_i f|=|\partial_i f|(M)$ denotes the total variation of the distributional derivative along $e_i$.}
For the Kohn-M\"uller model \eqref{eq:KM}, the elastic energy of the austenite part can be expressed as trace norm at the interface, i.e., with $u_0(\cdot):=u(0,\cdot)$,
\[[u_0]_{H^{1/2}_N((0,1))}^2:=\inf\left\{\int_0^\infty\int_0^1 |\nabla v(x)|^2\dx:\ v(0,x_2)=u_0(x_2),\ v\in W^{1,2}_{\text{loc}}((0,\infty)\times(0,1))\right\} \text{\ for \ }u\in\mathcal{A}_N,\]
and  for $u\in\mathcal{A}_P$
\[[u_0]_{H^{1/2}_P((0,1))}^2:=\inf\left\{\int_0^\infty\int_0^1 |\nabla v(x)|^2\dx:\ v(0,x_2)=u_0(x_2),\ v\in W^{1,2}_{\text{loc}}((0,\infty)\times(0,1)),\ v(\cdot, 0)=v(\cdot,1)\right\}.\]
We denote by $H^{1/2}_P((0,1))$ and $H^{1/2}_N((0,1))$ the spaces of functions in $L^2((0,1))$ for which the respective seminorms are finite. 
Note that the above are two different seminorms, and the choice of boundary conditions affects the value of the trace norm,
hence we use two different symbols for $H^{1/2}_N$ and $H^{1/2}_P$. 
Indeed, functions with jump discontinuities are not contained in $H^{1/2}$, hence no jump discontinuity between the two
boundary points is permitted in $H^{1/2}_P$, at variance with $H^{1/2}_N$ (see Remark \ref{rem:singlelam}). For later reference, we recall equivalent representations of the norms.
%Correspondingly, if the functionals are defined only
%on the functions ``inside'', there are two variants of the functional. This only becomes acute in the $F(v)$ formulation, I think.
%}
First (see, for example, \cite{giuliani-mueller:12}),
\begin{eqnarray}\label{eq:H1/2P}
[u_0]_{H^{1/2}_P{\com{((0,1))}}}^2=\int_0^1\int_{-\infty}^{+\infty}\frac{|u_0(z_1)-u_0(z_2)|^2}{|z_1-z_2|^2}\dz\text{\qquad for all\quad}u\in\mathcal{A}_P, 
 \end{eqnarray}
where we identify $u_0$ with its $1$-periodic extension; on the other hand, by Gagliardo's inequality (see, e.g., \cite[Chapter 15.3]{leoni}),
\begin{eqnarray}\label{eq:H1/2N}
[u_0]_{H^{1/2}_N{\com{(}}(0,1){\com{)}}}^2\sim\int_0^1\int_{0}^{1}\frac{|u_0(z_1)-u_0(z_2)|^2}{|z_1-z_2|^2}\dz\text{\qquad for all \quad}u\in\mathcal{A}_N.
\end{eqnarray}
Further, for $u\in H_P^{1/2}((0,1))$, the $H^{1/2}_P$-seminorm can be equivalently characterized in terms of the Fourier series. Precisely, if $u(x)=\sum_{k\in\Z}u_ke^{2\pi ikx}$, then 
\begin{eqnarray}
\label{eq:H1/2Fourier}
[u]_{H^{1/2}_P((0,1))}^2\sim\sum_{k\in\Z}|k||u_k|^2.
\end{eqnarray}
We note the following interpolation result.
\begin{lemma}\label{lemmah12}
%\begin{itemize}
%\item[(i)] There is a constant $c>0$ such that for all $u\in H^{1/2}_P((0,1))$ and $\psi\in H_P^{1/2}((0,1))\cap H^1((0,1))$,  
%\begin{eqnarray}\label{eq:interpolFourier1}
%\int_0^1u(x)\psi'(x)\dx\leq c [u]_{H^{1/2}_P((0,1))}[\psi]_{H^{1/2}_P((0,1))}.
%\end{eqnarray}
%\item[(ii)] 
There is a constant $c>0$ such that for all $v\in H^{1/2}_N((0,1))$ and $\psi\in H_P^{1/2}((0,1))\cap H^1((0,1))$,  
\begin{eqnarray}\label{eq:interpolFourier2}
\int_0^1v(x)\psi'(x)\dx\leq c [v]_{H^{1/2}_N((0,1))}[\psi]_{H^{1/2}_N((0,1))}.
\end{eqnarray}
%\end{itemize}
\end{lemma}
\begin{proof}
We first assume additionally that  $u\in H^{1/2}_P((0,1))$.
Let $u(x)=\sum_{k\in\Z}u_ke^{2\pi ikx}$ and $\psi(x)=\sum_{k\in\Z}\psi_ke^{2\pi ikx}$ be the Fourier series. Then by Cauchy-Schwarz and \eqref{eq:H1/2Fourier}
\begin{eqnarray*}
\left|\int_0^1u(x)\psi'(x)\dx\right|\lesssim\left|\sum_{k\in\Z}ku_k\psi_k\right|\lesssim[u]_{H^{1/2}_P((0,1))}[\psi]_{H^{1/2}_P((0,1))}.
\end{eqnarray*}
 We turn to the general case. Without loss of generality, we may assume that $\psi(0)=\psi(1)=0$. We extend $v$ and $\psi$ to $(0,2)$ by setting $v(1+x):=v(1-x)$ and $\psi(1+x):=-\psi(1-x)$ for $x\in(0,1)$. Then $v$, $\psi\in H^{1/2}_P((0,2))$ and by (i),
\begin{eqnarray*}
\int_0^1v(x)\psi'(x)\dx=\frac{1}{2}\int_0^2v(x)\psi'(x)\dx\leq c [v]_{H^{1/2}_P((0,2))}[\psi]_{H^{1/2}_P((0,2))}\sim [v]_{H^{1/2}_N((0,1))}[\psi]_{H^{1/2}_N((0,1))},
\end{eqnarray*}
where the last step follows directly from the definition by extending test functions for $v$ and $\psi$ on the smaller strip $(0,\infty)\times(0,1)$ (anti-)symmetrically to $(0,\infty)\times(0,2)$, and on the other hand restricting test functions on $(0,\infty)\times(0,2)$ to the smaller strip.
\end{proof}
Similarly, we define the $H^{1/2}_N$-seminorm for functions on rectangles $R\subset\R^2$ via
\[[v]_{H^{1/2}_N(R)}^2:=\inf \left\{\int_R\int_0^\infty |\nabla V(x)|^2\dx:\ V(x_1,x_2,0)=v(x_1,x_2),\ V\in W^{1,2}_{\text{loc}}( R\times(0,\infty))\right\},\]
and for the boundary of a cube $\Omega=(0,L)^3$, we set 
\[[v]_{H^{1/2}(\partial\Omega)}^2:=[v]_{H^{1/2}_N(\partial\Omega)}^2:=\inf \left\{\int_{\R^3\setminus\Omega} |\nabla V(x)|^2\dx:\ V=v\text{\ on\ }\partial\Omega,\ V\in W^{1,2}_{\text{loc}}(\R^3\setminus\Omega)\right\}. \]
We will use the following lemma.
\begin{lemma}\label{lem:H1/2average}
Set $R:=(0,a)\times(0,b)\subset\R^2$. Let $H:R\to\R$ and set $h:(0,b)\to\R$, $h(x_2):=\frac{1}{a}\int_0^aH(x_1,x_2)dx_1$. Then $[h]_{H^{1/2}_N((0,b))}\lesssim a^{-1/2}[H]_{H^{1/2}_N(R)}$.
\end{lemma}
\begin{proof}
Let $\delta>0$ be arbitrary, and let $\tilde{H}:R\times(0,\infty)\to\R$ be such that $\tilde{H}(x_1,x_2,0)=H(x_1,x_2)$ and $\int_{R\times(0,\infty)}|\nabla \tilde{H}|^2\dx\leq[H]_{H^{1/2}_N(R)}^2+\delta$. Define an extension $\tilde{h}$ for $h$ via $\tilde{h}(x_2,x_3):=\frac{1}{a}\int_0^a\tilde{H}(x_1,x_2,x_3)\dx_1$. Then the assertion follows by Jensen's inequality and arbitrariness of $\delta>0$.

\end{proof}
\section{Proof of the scaling laws for the Kohn-M\"uller model}
\label{secpfKM}
\begin{figure}
\centerline{%
\includegraphics[height=3.5cm]{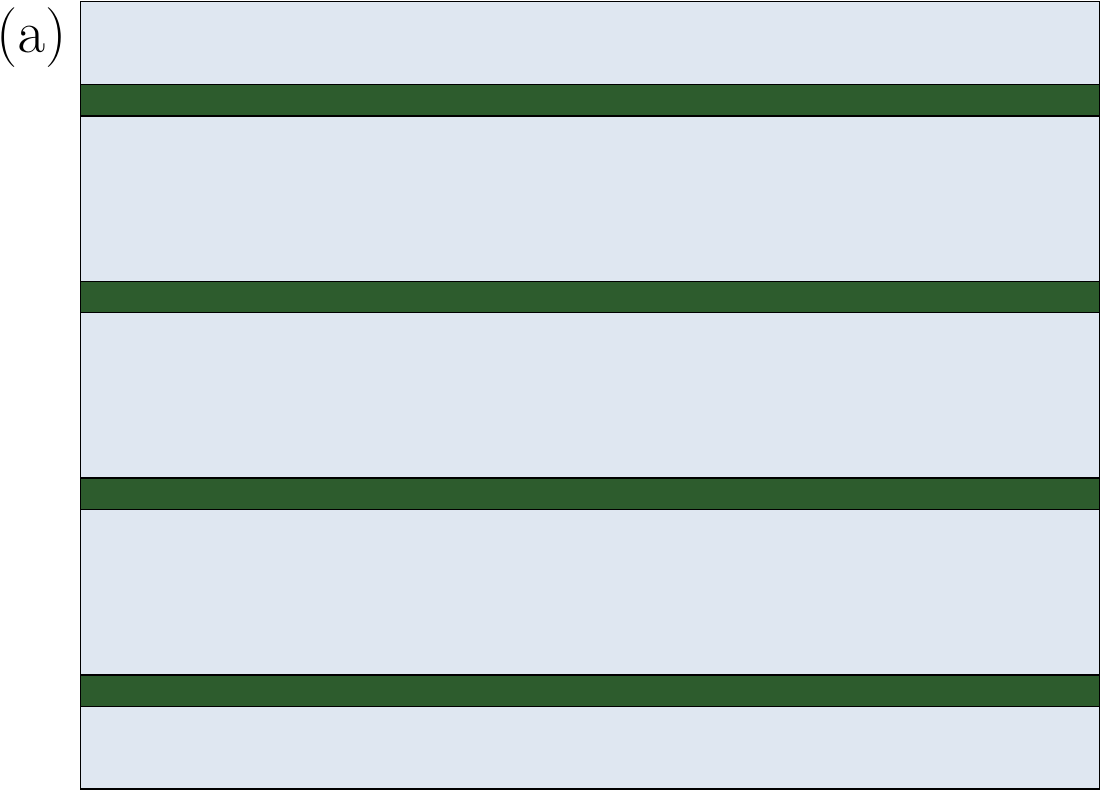}\hskip.8cm
 \includegraphics[height=3.5cm]{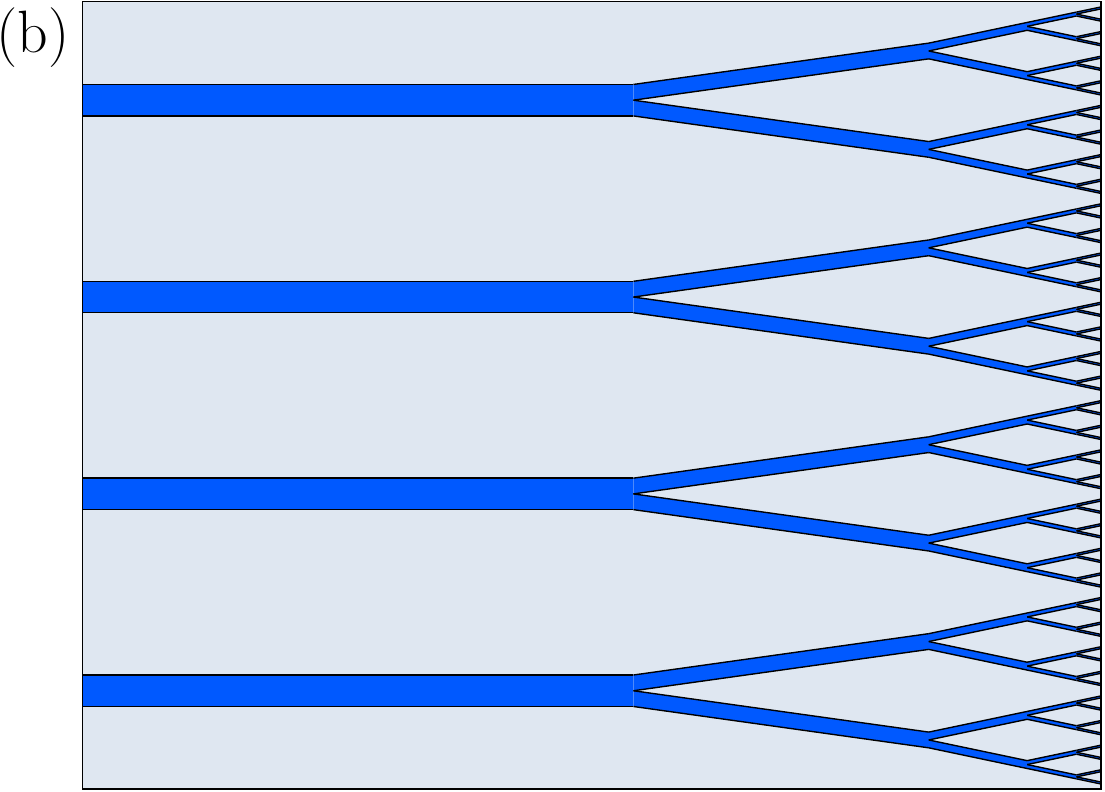}\hskip.8cm
 \includegraphics[height=3.5cm]{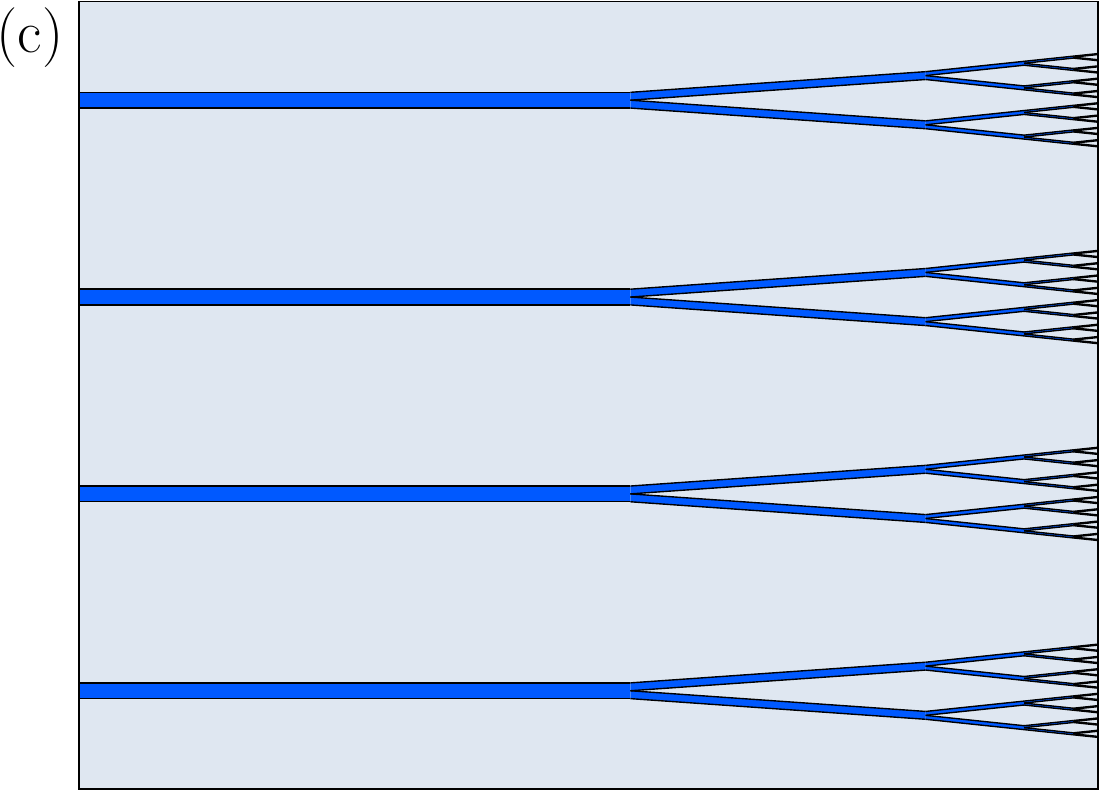}}
%\centerline{%
%\includegraphics[height=3.5cm]{fig-lam2-crop}\hskip.8cm
% \includegraphics[height=3.5cm]{fig-br2-disc-crop}\hskip.8cm
% \includegraphics[height=3.5cm]{fig-new2-disc-crop}}
 \caption{Sketch of the constructions used in the upper bound.}
 \label{figkmub}
\end{figure}
\subsection{Upper bound}\label{sec:ub}
We use the following construction for the proof of the upper bounds in \eqref{eq:scaling1} and, with slight modifications also in \eqref{eq:scaling2}. 
Our main new contribution in this section is the two-scale branching construction sketched in Figure \ref{figkmub}(c). An important building block for all constructions are the following two functions which also show the main difference between the two kinds of boundary conditions.

\paragraph{Uniform configuration/Single laminate.} 
\begin{itemize}
 \item[(i)] For Neumann boundary conditions, we consider a uniform configuration in the martensite part and set
 \[u_U(x_1,x_2):=\begin{cases}
 \theta x_2 &\text{\qquad if\quad}(x_1,x_2)\in[-L,0]\times[0,1],\\
 (1-x_1)\theta x_2&\text{\qquad if\quad}(x_1,x_2)\in[0,1]\times[0,1],\\
 0&\text{\qquad otherwise}.
    \end{cases}
  \]
   Then $u_U\in \mathcal{A}_N$ and 
   \[J(u_U)=\mu\theta^2 \int_0^1\int_0^1\left(x_2^2+(1-x_1)^2\right)\dx=\frac{2}{3} \mu\theta^2.\]
\item[(ii)] For periodic boundary conditions, we consider a single laminate in the martensite part and set
\begin{eqnarray}
\label{eq:singlelam}
u_{SL}(x_1,x_2):=\begin{cases}
             (-1+\theta)x_2&\text{\qquad if\quad}(x_1,x_2)\in[-L,0]\times[0,\theta],\\
\theta  (x_2-1)&\text{\qquad if\ }-L\leq x_1\leq 1,\ \max\{\theta,\ (1-\theta)x_1+\theta\}\leq x_2\leq 1,\\ 
\theta x_2(1-\frac{1}{(1-\theta)x_1+\theta})&\text{\qquad if\quad}0\leq x_1\leq 1,\ 0\leq x_2\leq (1-\theta)x_1+\theta,\\ 
0&\text{\qquad if\quad}x_1\geq 1.
             \end{cases}
 \end{eqnarray}
Then $u_{SL}\in \mathcal{A}_P$, and 
\[J(u_{SL})\lesssim \eps L+\mu\theta
^2\ln\frac{1}{\theta
}=\theta^2(\hat{\eps}L+\mu\ln\frac{1}{\theta
}).\]
%since 
%\begin{eqnarray*}
%\int_0^1\int_0^{(1-\theta)x_1+\theta}|\partial_1u|^2\dx&=&\int_0^1\int_0^{(1-\theta)x_1+\theta}\frac{(1-\theta)^2\theta^2x_2^2}{((1-\theta)x_1+\theta)^2}\\
%&=&\int_0^1\mbox{d}x_1\frac{1}{3}(1-\theta)^2\theta^2((1-\theta)x_1+\theta)\lesssim\theta^2,
%\end{eqnarray*}
%and
%\begin{eqnarray*}
%\int_0^1\int_0^{(1-\theta)x_1+\theta}|\partial_2u|^2\dx&=&\theta^2\int_0^1\int_0^{(1-\theta)x_1+\theta}(1-\frac{1}{(1-\theta)x_1+\theta})^2\dx\\
%&=&\theta^2(c+\int_0^1\frac{1}{(1-\theta)x_1+\theta}\mbox{d}x_1)\lesssim\theta^2\ln\frac{1}{\theta}.
%\end{eqnarray*}
\end{itemize}
\begin{remark}\label{rem:singlelam}
We note that, given the single laminate construction in the martensite part, the extension to the austenite part given in \eqref{eq:singlelam} yields the optimal scaling for the elastic energy. Precisely,
\begin{eqnarray}\label{eq:singlelamper}
[u_{SL}(0,\cdot)]_{H^{1/2}_P{\com{(}}(0,1){\com{)}}}^2\sim\theta^2\ln\frac{1}{\theta}. 
\end{eqnarray}
The logarithmic correction  term is due to the periodicity assumption. In contrast, we have
\begin{eqnarray}\label{eq:singlelamn}
[u_{SL}(0,\cdot)]_{H^{1/2}_N{\com{(}}(0,1){\com{)}}}^2\sim [u_{U}(0,\cdot)]_{H^{1/2}_N{\com{(}}(0,1){\com{)}}}^2\sim\theta^2. 
\end{eqnarray}
\end{remark}
\begin{proof}
To prove \eqref{eq:singlelamper} it remains to show a lower bound on the seminorm. For that, we use the equivalent form \eqref{eq:H1/2P}, we identify $u_{SL}(0,\cdot)$ with its $1$-periodic extension and estimate
\begin{eqnarray*}
[u_{SL}(0,\cdot)]_{H^{1/2}_P{\com{(}}(0,1){\com{)}}}^2&\geq&\int_{7/8}^1\int_{1+\theta}^{{\com{1+}}\theta+1/8}\frac{|u_{SL}(0,z_1)-u_{SL}(0,z_2)|^2}{|z_1-z_2|^2}\dz\\
&\geq& \left(\frac{\theta}{4}\right)^2\int_{7/8}^1\int_{1+\theta}^{{\com{1+}}\theta+1/8}\frac{1}{|z_1-z_2|^2}\dz\gtrsim\theta^2\ln\frac{1}{\theta}.
\end{eqnarray*}
To show \eqref{eq:singlelamn}, we first extend $u_{SL}(0,\cdot)$ to $[0,\theta]\times[0,1]$ by 
\begin{eqnarray*}
u(x_1,x_2):=\begin{cases}
(-1+\theta)x_2-x_1&\text{\qquad if\quad}0\leq x_2\leq\theta-x_1,\\
\theta(x_2-1)&\text{\qquad if\quad}\theta-x_1\leq x_2\leq 1.
\end{cases}
\end{eqnarray*}
Then $u(\theta,x_2)=u_U(0,x_2)-\theta$, and consequently, since $\int_0^\theta\int_0^1|\nabla u(x)|^2\dx\lesssim\theta^2$, 
we have $[u_{SL}]_{H^{1/2}_P(0,1)}^2\lesssim [u_{U}]_{H^{1/2}_N(0,1)}^2+\theta^2$ and $[u_{U}]_{H^{1/2}_N(0,1)}^2\lesssim [u_{SL}]_{H^{1/2}_P(0,1)}^2+\theta^2${\com{, which concludes the proof since 
$[u_{U}]_{H^{1/2}_N(0,1)}^2\sim\theta^2$ as shown in (i) above.}}
%where we note for completeness that  an extension with the optimal scaling properties is given by 
%\[\tilde{u}(x_1,x_2):=\begin{cases}
%\theta (1-x_1)x_2&\text{\qquad if\quad}0\leq x_1\leq ,1\\
%0&\text{\qquad otherwise.}
%\end{cases} \]
\end{proof}
All remaining constructions are special cases of a two-scale branching construction, which relies on the following lemma. The latter is a straight-forward generalization of {\com{a}} construction {\com{from}} \cite[Lemma 2.3]{kohn-mueller:94}.
 \begin{lemma}
\label{lem:branch}
Suppose that $0<\theta\leq\frac{1}{2}$, and let $h$, $\eta$, $\ell>0$ be such that $\theta h\leq\eta\leq h$. Then there is a function $b:=b^{(h,\eta,\ell)}:(-\ell,0)\times(0,h)\to\R$ with the following properties:
\begin{itemize}
\item[(i)] $b(x_1,0)=0$, $b(x_1,h)=h\theta-\eta$,
\[b(-\ell,x_2)=\begin{cases}
\theta x_2&\text{\quad if\ }0\leq x_2\leq (h-\eta)/2,\\
(-1+\theta)x_2+(h-\eta)/2&\text{\quad if\ } (h-\eta)/2\leq x_2\leq(h+\eta)/2,\\
\theta x_2-\eta&\text{\quad if\ } (h+\eta)/2\leq x_2\leq h,\\
\end{cases} \]
and $b(-\ell,x_2)=N(h\theta-\eta)+b(-\ell,\tilde{x}_2)$ if $x_2=Nh+\tilde{x}_2$ with $N\in\N$ and $\tilde{x_2}\in[0,h)$,
\item[(ii)] $b(0,x_2)=\frac{1}{2}b(-\ell,2x_2)$ if $0\leq x_2\leq h/2$ and $\partial_2b(0,x_2)$ is $h/2$-periodic,
\item[(iii)] $\partial_2b\in\{\theta,-1+\theta\}$ almost everywhere,
\item[(iv)] $\|\partial_1b\|_{L^2((-\ell,0)\times(0,h))}^2\leq C\eta^2 (h-\eta)/\ell$, where $C$ does not depend on $h$, $\ell$, $\theta$ or $\eta$, 
\item[(v)] $\int_{(-\ell,0)\times(0,h)}|\partial_2\partial_2b|\leq 4\ell$
\item[(vi)] $\int_{(-\ell,0)\times(0,h)}|\partial_1\partial_2b|\leq 4\eta$.
\end{itemize}
\end{lemma}
\begin{proof}
The function $b$ is sketched in Figure \ref{figlembranch} {\BZ{(right)}}. By (ii), it suffices to construct $b$ for $x_2\leq h/2$, namely
\begin{eqnarray*}
b(x_1,x_2):=\begin{cases}
\theta x_2&\text{\ if\ }0\leq x_2\leq \frac{h-\eta}4,\\
(\theta-1) x_2 +\frac{h-\eta}4&\text{\ if\ }\frac{h-\eta}{4}\leq x_2\leq \frac{h{\com{+}}\eta}{4}+\frac{\eta x_1}{2\ell},\\
\theta x_2-\frac{\eta x_1}{2\ell}{\com{-\frac{\eta}{2}}}&
\text{\ if\ }\frac{h{\com{+}}\eta}{4}+\frac{\eta x_1}{2\ell}\leq x_2\leq 
\frac{{\com{h}}}2+\frac{\eta x_1}{2\ell},\\
(\theta-1) x_2+\frac{h-\eta}{2}&\text{\ if\ }
\frac{{\com{h}}}2+\frac{\eta x_1}{2\ell}
\leq x_2\leq h/2.
\end{cases}
\end{eqnarray*}
and to evaluate all integrals. We remark that all conditions except the last one would also be satisfied by the construction with a connected minority phase. 
The last one is not important for the Kohn-Müller model, but will become relevant for the dislocation model.
\begin{eqnarray*}
\hat b(x_1,x_2):=\begin{cases}
\theta x_2&\text{\ if\ }0\leq x_2\leq \frac{h-\eta}4(1-\frac {x_1}\ell),\\
(\theta-1) x_2 +\frac{h-\eta}4(1-\frac {x_1}\ell)
&\text{\ if\ }\frac{h-\eta}4(1-\frac {x_1}\ell)\leq x_2\leq 
\frac{h-\eta}4(1-\frac {x_1}\ell)+\frac\eta2,\\
\theta x_2-\frac{\eta }{2}&
\text{\ if\ }
\frac{h-\eta}4(1-\frac {x_1}\ell)+\frac\eta2\leq x_2\leq \frac h2,
\end{cases}
\end{eqnarray*}
extended so that $\partial_2\hat b$ is $h/2$-periodic in $x_2$.
\end{proof}
\begin{figure}
\centerline{%
\includegraphics[height=4cm]{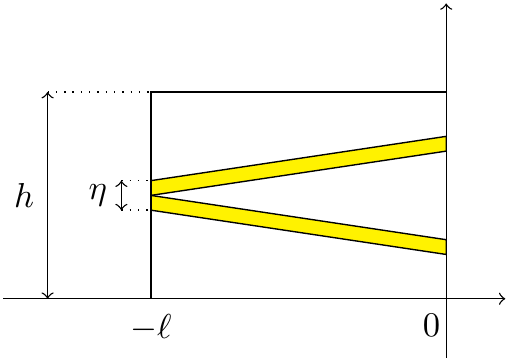}\hskip1cm
\includegraphics[height=4cm]{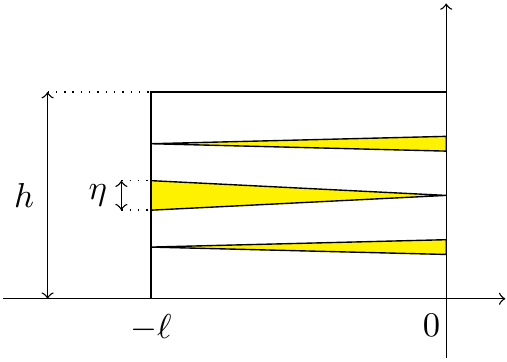}}
 \caption{Sketch of the construction of Lemma \ref{lem:branch}.
 Left: the connected pattern described by $\hat b$, which has a longer vertical component of the interface and does not 
 satisfy (vi). Right: the disconnected pattern
 of $b$, which fulfills all stated properties.
 }
 \label{figlembranch}
\end{figure}

%\paragraph{Branching.} If $\hat{\eps}<1/L^2$, then the variant of the Kohn-M\"uller branching constuction \cite[Lemma 2.3]{kohn-mueller:94}
%for unequal volume fractions  is admissible in the martensite part \cite{diermeier:10,zwicknagl:14}. 
%This is the deformation sketched in Figure~\ref{figkmub}(b).
%This gives a function $u_B\in\mathcal{A}_P\subset\mathcal{A}_N$ with $u_B(0,x_2)=0$, which then can be extended to $[0,\infty)\times(0,1)$ by $u_B(x_1,x_2)=0$. Summarizing, 
%\[J(u_B)\lesssim\eps^{2/3}L^{1/3}\theta
%^{2/3}=\theta^2\hat{\eps}^{2/3}L^{1/3}.\]
%

\begin{figure}
\centerline{%
 \includegraphics[height=3.5cm]{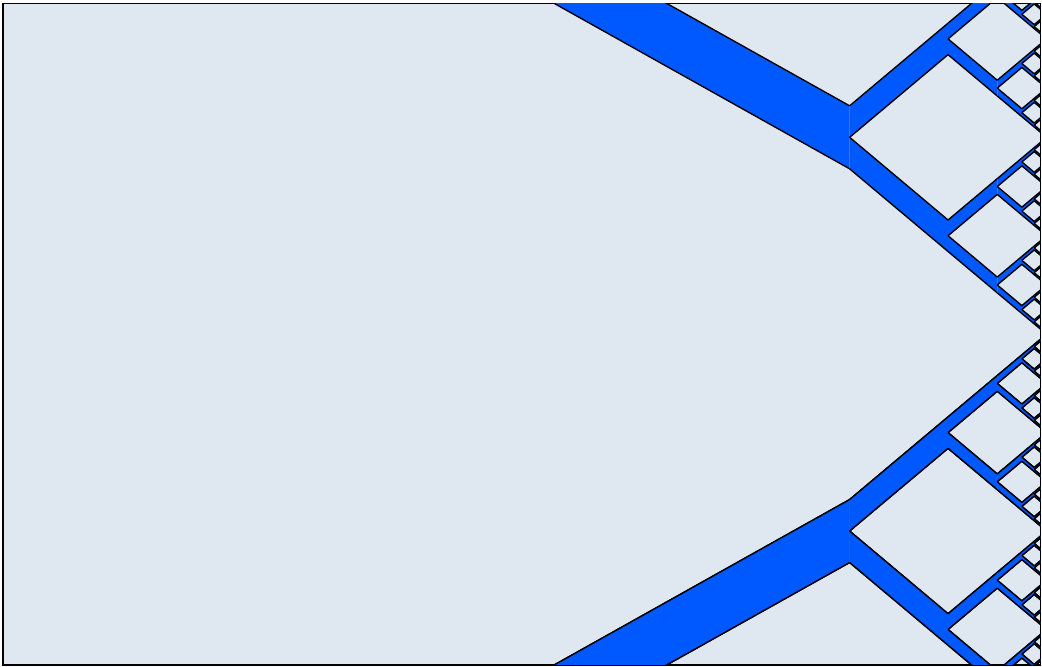}\hskip.8cm
 \includegraphics[height=3.5cm]{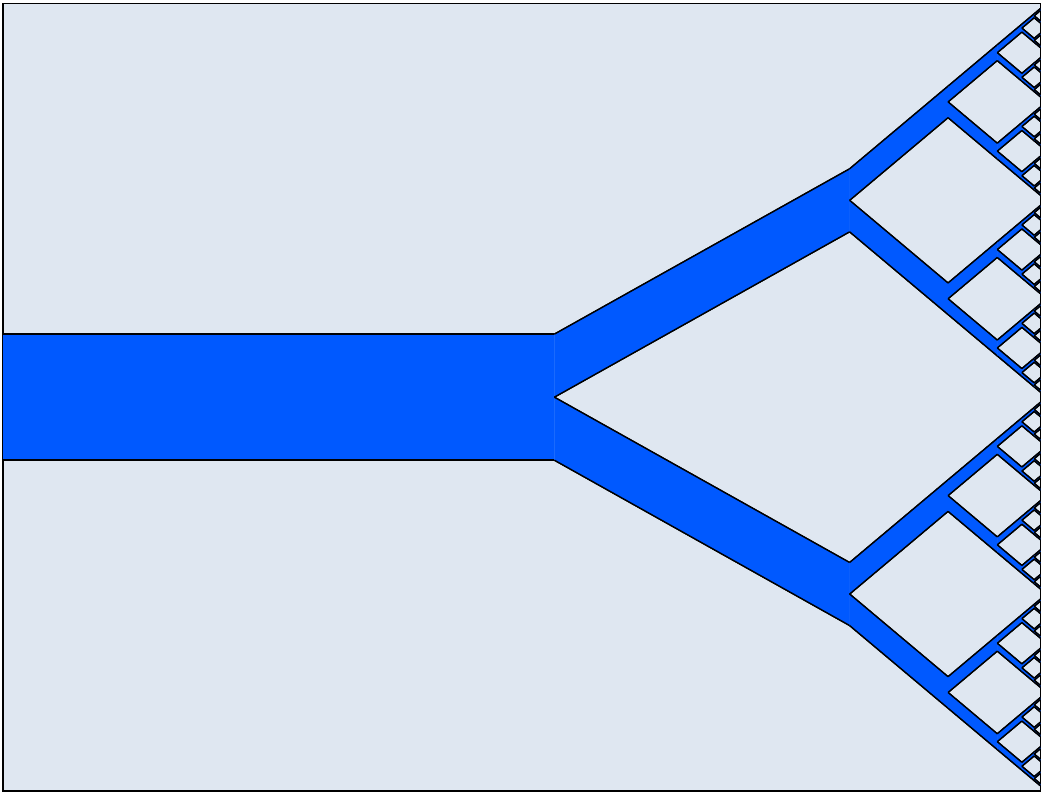}\hskip.8cm}
 \caption{Sketch of the truncated branching construction
 for Neumann boundary data (left) and periodic boundary data (right).}
 \label{figtruncb}
\end{figure}
%\paragraph{Truncated branching.} 
%\begin{itemize}
%\item[(a)] 
%For Neumann boundary conditions, we use the truncated branching construction from \cite[Theorem 1]{zwicknagl:14} extended by $u_{TB}(x_1,x_2)=0$ for $x_1\geq 0$, see Figure \ref{figtruncb}. This yields a function $u_{TB}\in\mathcal{A}_N$ with 
%\[J(u_{TB})\lesssim \eps^{1/2}\theta=\theta^2\hat{\eps}^{1/2}.\]
%\item[(b)] For periodic boundary conditions, the test function has to be slightly modified, and we build a test function $u_{TB}^{(P)}\in\mathcal{A}_P$ as follows: On $[-\ell,0]\times [0,1]$ we employ the standard branching construction, where $\ell\sim\hat{\eps}^{-1/2}\theta^2
%$ is chosen such that $\partial_2u_{TB}^{(P)}$ has at most two jumps at $x_1=\ell$ (see \cite[Theorem 1]{zwicknagl:14} for details). If $\ell<L$, then we set $u_{TB}^{(P)}(x_1,x_2):=u_{TB}^{(P)}(-\ell,x_2)$ for all $x_1\leq-\ell$ and all $x_2\in[0,1]$, and we set $u_{TB}^{(P)}(x_1,x_2):=0$ if $x_1\geq 0$. Then 
%\[J(u_{TB}^{(P)})\sim\eps^{1/2}\theta+\eps L=\theta^2(\hat{\eps}^{1/2}+\hat{\eps}L).\] 
%\end{itemize}
\paragraph{Two-scale branching.}
%We consider the regime
%\[\hat\eps\leq
%\min\{\frac{\mu}{L}\ln\left(3+\frac{\eps}{\mu^3L\theta^2}\right),\ 
%\frac{\mu^3 L}{\theta}. \]
We present a new construction which as special cases comprises test functions from the literature, in particular laminated structures, the single laminate and branched patterns (see below). 
 Let $N\in\N$, {$0<\ell\leq L$ and $\theta\leq h\leq 1$ be such that $\ell=L$ if $N>1$. We construct a deformation $u_{TSB}$ in three steps:\\}
 {\em{Step 1: Construction for $-\ell\leq x_1\leq 0$: }}
 %be such that 
% \begin{eqnarray}\label{eq:choiceN}
% N\sim \max\{\tilde{N},\ 1\},\text{\qquad where\quad} \tilde{N}:=\min\left\{\left(\frac{1}{\hat{\eps}L^2}\right)^{1/3},\ \left(\frac{\mu}{\hat{\eps}L}\ln(\frac{1}{\theta^2})\right)^{1/2},\ \left(\frac{\mu}{\hat{\eps}L}\ln(3+\frac{\hat{\eps}}{\mu^3 L})\right)^{1/2}\right\},
 %\end{eqnarray}
 % and let $\theta\leq h\leq 1$. 
 We construct an admissible function $u_{TSB}$ that is periodic in $x_2$-direction with period $1/N$, and describe the  construction on ${\com{[}}-\ell,0{\com{]}}\times{\com{[}}0,1/N{\com{]}}$. 
%\[u_{TSB}(x_1,x_2):=\begin{cases}
%\theta x_2&\text{\qquad if\quad}x_2\in[0,\frac{1-\theta}{2N}]\\
%(1-\theta)(\frac{1}{2N}-x_2)&\text{\qquad if\quad}x_2\in[\frac{1-\theta}{2N},\frac{1+\theta}{2N}]\\
%\theta x_2-\frac{\theta}{N}&\text{\qquad if \quad}x_2\in[\frac{1+\theta}{2N},\frac{1}{N}].
%\end{cases} \]
%Then the total energy contribution in $R_0$ is bounded above by
%\begin{eqnarray}\label{eq:TSB1}
%2\eps LN\lesssim 
%\theta^2\hat\eps L + 
%\theta^2(\mu\hat{\eps} L)^{1/2}\ln^{1/2}\left(3+\frac{\hat{\eps}}{\mu^3 L}\right). 
%\end{eqnarray}
%On $(-\ell,0)\times(0,1/N)$ 
We set
\[u_{TSB}(x_1,x_2):=\begin{cases}
\theta x_2&\text{\qquad if\quad}x_2\in[0,\frac{1-h}{2N}],\\
\theta (x_2-\frac{1}{N})&\text{\qquad if\quad}x_2\in[\frac{1+h}{2N},\frac{1}{N}],
\end{cases} \]
and construct a branching function on the rectangle $R_1:={\com{[}}-{\cm{\ell}},0]\times{\com{[}}\frac{1-h}{2N},\frac{1+h}{2N}]$ as follows: Subdivide $R_1$ into rectangles
\[R_{ij}:={\com{[}}-3^{-i}\ell,-3^{-(i+1)}\ell{\com{]}}\times{\com{[}}\frac{1-h}{2N}+\frac{jh}{2^iN},\frac{1-h}{2N}+\frac{(j+1)h}{2^iN}{\com{]}},\qquad i=0,1,\dots,\text{\quad and\quad}j=0,\dots,2^i-1, \]
and consider the bottom rectangles $R_{i0}$ first. On level $i=0,1,\dots$ choose $\ell_i:=\frac{2}{3}3^{-i}\ell$, $h_i:=\frac{h}{2^iN}$ and $\eta_i:=\frac{\theta}{2^iN}$ and note that the assumptions of Lemma \ref{lem:branch} are satisfied. We use the function $b^{(h_i,\eta_i,\ell_i)}$ from Lemma \ref{lem:branch} and set
\[u_{TSB}(x_1,x_2):=b^{(h_i,\eta_i,\ell_i)}(x_1+(3^{-i}\ell-\ell_i), x_2-\frac{1-h}{2N})+\frac{1-h}{2N}\theta. \]
The function $u_{TSB}$ is then extended to a continuous function on $\cup_{j}R_{ij}$ such that $\partial_2u_{TSB}$ is $h/(2^iN)$-periodic in $x_2$-direction. By Lemma \ref{lem:branch} this yields a continuous function $u_{TSB}$ on $(-\ell,0)\times(0,1)$ with $\partial_2u_{TSB}\in\{\theta,-1+\theta\}$ almost everywhere, and 
\begin{eqnarray}
\nonumber
\int_{-\ell}^0\int_0^1|\partial_1u_{TSB}|^2\dx+\eps|\partial_2\partial_2 u_{TSB}| 
&\lesssim&\sum_{i=0}^\infty \left\{\left(\frac{3}{4}\right)^i\frac{(h-\theta)}{\ell}\frac{\theta^2}{N^2}+\left(\frac{2}{3}\right)^i\eps \ell N\right\}
\lesssim\frac{(h-\theta)}{\ell}\frac{\theta^2}{N^2}+\eps \ell N.\\
\qquad \label{eq:TSB2}
\end{eqnarray}
%We extend the test function to the austenite part similarly as for the laminates. 
{\em Step 2: Construction for $x_1\geq 0$:}
For $0\leq x_1\leq\frac{1-h}{2N}$ and $0\leq x_2\leq\frac{1}{N}$, we set $u_{TSB}(x_1,x_2):=\frac{1}{N}u(Nx_1,Nx_2)$ with
\[u(x_1,x_2):=\begin{cases}
\theta x_2&\text{\qquad if\quad}0\leq x_2\leq -x_1+\frac{1-h}{2},\\
\theta(1-\frac{1}{2x_1+h})(x_2-\frac{1}{2})&\text{\qquad if\quad}-x_1+\frac{1-h}{2}\leq x_2\leq x_1+\frac{1+h}{2},\\
\theta(x_2-1)&\text{\qquad if\quad}x_1+\frac{1+h}{2}\leq x_2\leq 1,
\end{cases} \]
set $u_{TSB}(x_1,x_2):=0$ if $x_1\geq\frac{1-h}{2N}$, and extend it periodically in $x_2$-direction with period $\frac{1}{N}$. Note that for $h=1$, we have $u(x_1,x_2)=0$ if $x_1\geq 0$. We obtain\begin{eqnarray}\label{eq:TSB3}
\mu\int_0^\infty\int_0^1|\nabla u_{TSB}(x)|^2\dx\lesssim\mu\frac{\theta^2}{N}\ln\frac{1}{h}. 
\end{eqnarray}
Summarizing, if $\ell=L$, we have constructed a function $u_{TSB}\in\mathcal{A}_D\subset\mathcal{A}_N$ such that (see \eqref{eq:TSB2} and \eqref{eq:TSB3})
\begin{eqnarray}\label{eqnenergytsb}
J(u_{TSB})\lesssim \theta^2\big(\frac{(h-\theta)}{L}\frac{1}{N^2}+\hat{\eps} LN+\frac{\mu}{N}\ln\frac{1}{h}\big). 
\end{eqnarray}
We remark that the term proportional to $\mu$ disappears for $h=1$.

{\em Step 3: Construction for $-L\leq x_1\leq\ell$:} It remains to consider the case $\ell<L$. Recall that we allow $\ell\leq L$ only if $N=1$ (this will be needed in case of Neumann boundary conditions below). We distinguish between Dirichlet and Neumann boundary conditions. To construct a periodic function $u_{TSB}^{(P)}\in\mathcal{A}_P$, we extend $u_{TSB}$ constantly in $x_1$, i.e.,
\[u_{TSB}^{(P)}(x_1,x_2):=u_{TSB}(-\ell,x_2)\text{\qquad for all\quad}-L\leq x_1\leq-\ell. \]
We obtain an additional energy contribution
\begin{eqnarray}
\label{eq:TSB4D}
\eps\int_{(-L,-\ell)\times(0,1)}|\partial_2\partial_2u_{TSB}^{(P)}|\lesssim(L-\ell)\eps N.
\end{eqnarray}
In case of Neumann boundary conditions, we introduce an interpolation layer analogously to the truncated branching construction from \cite[Theorem 1]{zwicknagl:14}.  Precisely, we set  (if $L\leq 2\ell$, we use the restriction of the function)
\[u_{TSB}^{(N)}(x_1,x_2):=\begin{cases}
(-1+\theta)x_2&\text{\ if\ }0\leq x_2\leq\frac{\theta}{2\ell}x_1+\theta,\\
\theta x_2-\frac{\theta}{2\ell}x_1-\theta&\text{\ if\ }\frac{\theta}{2\ell}x_1+\theta\leq x_2\leq-\frac{\theta}{2\ell}x_1+1-\theta,\\
(-1+\theta)x_2-\frac{\theta}{\ell}x_1+1-2\theta&\text{\ if\ }-\frac{\theta}{2\ell}x_1+1-\theta\leq x_2\leq 1,
\end{cases} \]
and extend $u_{TSB}^{(N)}$ constantly in $x_1$ for $x_1\leq -2\ell$, i.e., 
\[u_{TSB}^{(N)}(x_1,x_2):=u_{TSB}^{(N)}(-2\ell,x_2)=\theta x_2\text{\qquad if \quad}-L\leq x_1\leq -2\ell. \]
Therefore, the additional energy contribution is estimated above by
\begin{eqnarray}
\label{eq:TSB4N}
\int_{(-L,-\ell)\times(0,1)}|\partial_1u_{TSB}^{(N)}(x)|^2\dx+\eps|\partial_2\partial_2u_{TSB}^{(N)}|\lesssim\theta^2(\frac{1}{\ell}+\hat{\eps}\ell).
\end{eqnarray}
Summarizing, if $\ell<L$ and $N=1$, we have constructed functions $u_{TSB}^{(P/N)}$ with (see \eqref{eq:TSB2}, \eqref{eq:TSB3} and \eqref{eq:TSB4D} respectively \eqref{eq:TSB4N})
\begin{eqnarray}
\label{eq:TSBsumD}
J(u_{TSB}^{(P)})\lesssim\theta^2\big(\frac{(h-\theta)}{\ell}+\hat{\eps}L+\mu\ln\frac{1}{h}\big),
\end{eqnarray}
and 
\begin{eqnarray}
J(u_{TSB}^{(N)})\lesssim\theta^2\big(\frac{(h-\theta)}{\ell}+\mu\ln\frac{1}{h}+\hat{\eps}\ell+\frac{1}{\ell}\big).
\label{eq:TSBsumN}
\end{eqnarray}
As above, the term proportional to $\mu$ disappears for $h=1$.
We note for later reference that optimizing the upper bounds for the total energies (see \eqref{eqnenergytsb}, \eqref{eq:TSBsumD} and \eqref{eq:TSBsumN}) in $h$ subject to the constraint $\theta\leq h\leq 1$ yields 
\begin{eqnarray}\label{eq:hopt}
h=\min\{1,\max\{\theta, \mu LN\}\}.
\end{eqnarray}
\begin{remark}
\label{rem:trunc}
In order to reuse this construction in the upper bound for the dislocations, we will need a slight modification of Step 1, similarly to constructions from \cite{schreiber:94,conti-ortiz:05,zwicknagl:14,chan-conti:14-1,Kohn-Wirth:14-1}, i.e., we stop at some finite level $I\in\N$ such that the slopes of the interfaces remain bounded by a constant, i.e., $\eta_I/\ell_I\sim 1$, which is equivalent to
\begin{eqnarray}
\left(\frac{3}{2}\right)^I\sim\frac{N\ell}{\theta}.
\label{eq:abbruch}
\end{eqnarray}
We note that if such an $I\in\N$ exists, then 
\begin{eqnarray}
\label{eq:surface}
\eps\int_{-\ell}^{3^{-(I+1)}\ell}|D\partial_2 u_{TSB}|\lesssim\eps\int_{-\ell}^{3^{-(I+1)}\ell}|\partial_2\partial_2 u_{TSB}|\lesssim\eps\ell N.
\end{eqnarray}
%estimate $|D\partial_2 u|$. 
%If $h_i\lesssim \ell_i$, 
%$|D\partial_2 u|$ is controlled by $|\partial_2\partial_2 u|$.
%In turn, this occurs for all $i\le I:=\ln (\ell N/h)/\ln(3/2)$,
%and in particular for all $x_1<3^{-I}\ell$.
%
%We define $\hat\delta:= 3^{-I}\ell =\ell \exp(-I\ln 3)=
%\ell (\ell N/h)^{-(\ln 3)/(\ln 3/2)}$.
%Notice that we can redo the entire construction with $2\sqrt2$ instead of $3$. Then
%$\hat\delta=\ell (\ell N/h)^{-3}= \ell (\hat\eps N h)^{3/2}=h^2\hat \eps=\mu^3L$.

\end{remark}

We are now in the position to prove the upper bound of Theorem \ref{th:1}. We point out that $\ln^{1/2}\frac{1}{\theta}\sim\ln^{1/2}\frac{1}{\theta^2}$, and thus, we may replace $\ln^{1/2}\frac{1}{\theta}$ by $\ln^{1/2}\frac{1}{\theta^2}$ in the scaling law.
\begin{proposition}\label{prop:KMub}
There are constants $c_N$, $c_P>0$ such that or all $\eps$, $\mu$, $L>0$ and all $\theta{\com{\in(0,\frac{1}{2}]}}$, we have with $\hat{\eps}:=\eps/\theta^2$,
\begin{eqnarray}\label{eq:scaling1ub}
\min_{u\in\mathcal{A}_N} J(u)\leq c_N \theta^2\min\left\{\mu,\,\hat{\eps}^{1/2},\,\hat{\eps}^{2/3}L^{1/3},\,\,\left(\hat{\eps} L\mu\right)^{1/2}\ln^{1/2}\left(\frac{1}{\theta^2}\right),\, \,(\hat{\eps} L\mu)^{1/2}\ln^{1/2}\left(3+\frac{\hat{\eps}}{\mu^3 L}\right)\right\}
\end{eqnarray}
and 
\begin{eqnarray}\label{eq:scaling2ub}
\min_{u\in\mathcal{A}_P} J(u)\leq c_P\theta^2 \max\left\{\hat{\eps} L,\ \min\left\{\hat{\eps}^{2/3}L^{1/3},\, \left(\hat{\eps} L\mu\right)^{1/2}\ln^{1/2}\left(\frac{1}{\theta^2}\right),\, (\hat{\eps} L\mu)^{1/2}\ln^{1/2}\left(3+\frac{\hat{\eps}}{\mu^3 L}\right)\right\}\right\}.
\end{eqnarray}
\end{proposition}
\begin{proof}
%The constructions listed at the beginning of this section yield functions with the appropriate scaling properties. It remains to check that the constructions are admissible in the respective parameter regimes. We will frequently use that $\ln(3+\frac{\hat{\eps}}{\mu^3 L})\geq c$ for all choices of $\eps$, $\ell$, $\mu$ and $\theta$.
%We consider the Neumann case $\mathcal{A}_N$ first.
We prove \eqref{eq:scaling1ub} first. We use mostly variants of the two-scale-branching construction (TSB) and provide appropriate choices for the parameters $\ell$, $h$ and $N$ in the various regimes.
\begin{itemize}
\item[(i)] The uniform function $u_U$ 
%and the truncated branching $u_{TB}$ are 
is always admissible, and thus, $\min_{u\in\mathcal{A}_N}J(u)\leq c_1\theta^2\mu$.
\item[(ii)] For any given parameters we obtain an admissible function as restriction to $(-L,\infty)\times(0,1)$ of the TSB with $N=1$, $h=1$ and $\ell=\hat{\eps}^{-1/2}$. This yields the truncated branching construction from \cite[Theorem 1]{zwicknagl:14}, which is sketched in Figure \ref{figtruncb} (left). Thus (see also \eqref{eq:TSBsumN}), we have $\min_{u\in\mathcal{A}_N}J(u)\leq c_1\theta^2\hat{\eps}^{1/2}$.
\item[(iii)] In the regime in which the minimum in \eqref{eq:scaling1ub} is attained by $\hat{\eps}^{2/3}L^{1/3}$, we in particular have $\hat{\eps}^{2/3}L^{1/3}\leq\hat{\eps}^{1/2}$. Thus, we may choose $N\sim(\hat{\eps}L^2)^{-1/3}$, $h=1$ and $\ell=L$ in TSB. This is the variant of the Kohn-M\"uller branching constuction \cite[Lemma 2.3]{kohn-mueller:94} for unequal volume fractions as given in \cite{diermeier:10,zwicknagl:14}. This yields a function $u_B$ such that (see \eqref{eqnenergytsb}) $J(u_B)\lesssim \theta^2\hat{\eps}^{2/3}L^{1/3}$. 
The deformation is sketched in Figure~\ref{figkmub}(b).
%If the minimum in \eqref{eq:scaling1ub} is realized by $\hat{\eps}^{2/3}L^{1/3}$, then in particular $\hat{\eps}^{2/3}L^{1/3}\leq \hat{\eps}^{1/2}$, and thus, the branching construction $u_{B}$ is admissible.
\item[(iv)] If the minimum in \eqref{eq:scaling1ub} is attained by $\left(\hat{\eps} L\mu\right)^{1/2}\ln^{1/2}\left(\frac{1}{\theta^2}\right)$, then in particular $\left(\hat{\eps} L\mu\right)^{1/2}\ln^{1/2}\left(\frac{1}{\theta^2}\right)\leq\mu$, and hence, we may choose $N\sim\frac{\mu^{1/2}\ln^{1/2}\frac{1}{\theta^2}}{\hat{\eps}^{1/2}L^{1/2}}$, $\ell=L$ and $h=\theta$. This yields the classical laminate construction $u_{Lam}$, for which $J(u_{Lam})\lesssim\theta^2(\hat{\eps}L\mu)^{1/2}\ln^{1/2}\frac{1}{\theta^2}$, where we used that $\ln^{1/2}\frac{1}{\theta}\sim\ln^{1/2}\frac{1}{\theta^2}$. The function is sketched in Figure \ref{figkmub}(a).
\item[(v)] Finally, in the regime in which the minimum in \eqref{eq:scaling1ub} is attained by $(\hat{\eps} L\mu)^{1/2}\ln^{1/2}\left(3+\frac{\hat{\eps}}{\mu^3 L}\right)$, we use TSB with $N\sim\frac{\mu^{1/2}\ln^{1/2}\left(3+\frac{\hat{\eps}}{\mu^3 L}\right)}{\hat{\eps}^{1/2}L^{1/2}}$, $\ell=L$ and $h=\frac{L^{1/2}\mu^{3/2}\ln^{1/2}\left(3+\frac{\hat{\eps}}{\mu^3 L}\right)}{\hat{\eps}^{1/2}}$. Note that $h\sim N\mu L$ as motivated by \eqref{eq:hopt}. This choice is indeed admissible: First, $N\gtrsim 1$ follows from $(\hat{\eps} L\mu)^{1/2}\ln^{1/2}\left(3+\frac{\hat{\eps}}{\mu^3 L}\right)\leq \mu$. Second, $\theta\leq h$ since $(\hat{\eps} L\mu)^{1/2}\ln^{1/2}\left(3+\frac{\hat{\eps}}{\mu^3 L}\right)\leq (\hat{\eps} L\mu)^{1/2}\ln^{1/2}\frac{1}{\theta^2}$ implies that $h\geq\frac{\mu^{3/2}L^{1/2}}{\hat{\eps}^{1/2}}\geq \theta$. Finally, $h\leq 1$ since $(\hat{\eps} L\mu)^{1/2}\ln^{1/2}\left(3+\frac{\hat{\eps}}{\mu^3 L}\right)\leq \hat{\eps}^{2/3}L^{1/3}$ implies that $h\leq h\ln\left(3+\frac{\hat{\eps}}{
\mu^3 L}\right)\leq 1$. Thus, we obtain a function $u_{TSB}$ for which (see \eqref{eqnenergytsb})
\begin{eqnarray*}
J(u_{TSB})&\lesssim&\theta^2(\hat{\eps}L\mu)^{1/2}\left(\ln^{-1/2}\left(3+\frac{\hat{\eps}}{\mu^3 L}\right)+\ln^{1/2}\left(3+\frac{\hat{\eps}}{\mu^3 L}\right)+\frac{\ln\left(\frac{\hat{\eps}}{\mu^3 L\ln\left(3+\frac{\hat{\eps}}{\mu^3 L}\right)}\right)}{\ln^{1/2}\left(3+\frac{\hat{\eps}}{\mu^3 L}\right)}\right)\\
&\lesssim&\theta^2(\hat{\eps} L\mu)^{1/2}\ln^{1/2}\left(3+\frac{\hat{\eps}}{\mu^3 L}\right).
\end{eqnarray*}
This function is sketched in Figure \ref{figkmub}(c).
%Comparing to $\hat{\eps}^{2/3}L^{1/3}$ and $(\hat{\eps}L\mu)^{1/2}\ln^{1/2}\frac{1}{\theta}$
% we obtain $\mu^3 L\lesssim\hat\eps\lesssim\mu^3/(\theta^2L)$
%(for the second inequality one uses that $\ln(3+1/\theta)$ and $\ln (3+1\theta^2)$ are equivalent from the viewpoint of scaling).
%Therefore the two-scale branching construction $u_{TSB}$ is admissible. Further, comparing with 
%$\mu$ we obtain  $\eps\leq\mu^3 L$, therefore the first term in (\ref{eqnenergytsb}) can be neglected.
 \end{itemize}
We proceed similarly to prove \eqref{eq:scaling2ub}.
\begin{itemize}
\item[(i)] Suppose that the minimum in \eqref{eq:scaling2ub} is attained by $\hat{\eps}^{2/3}L^{1/3}$. There are two possibilities. If $\hat{\eps}^{2/3}L^{1/3}\geq\hat{\eps}L$, then the branching construction $u_B$ is admissible (see (iii) above). Otherwise, we have $\hat\eps^{1/2}\leq\hat{\eps}L$. Then we choose $N=1$, $h=1$ and $\ell=\hat{\eps}^{-1/2}$  in TSB. This yields a periodic variant $u_{TB}^{(P)}$ of the truncated branching construction  from \cite[Theorem 1]{zwicknagl:14} with (see \eqref{eq:TSBsumD}) $J(u_{TB}^{(P)})\lesssim\theta^2(\hat{\eps}^{1/2}+L\hat{\eps})\lesssim\theta^2\hat{\eps}L$. This function is sketched in Figure \ref{figtruncb} (right).
%
%the periodic truncated branching construction $u_{TB}^{(P)}$ has energy $\sim \theta^2(\hat\eps L+\hat\eps^{1/2})\leq 2\theta^2\hat\eps L$.
\item[(ii)] Suppose that the minimum in \eqref{eq:scaling2ub} is attained by $\left(\hat{\eps} L\mu\right)^{1/2}\ln^{1/2}(\frac{1}{\theta^2})$. Again, there are two possibilities: If $\left(\hat{\eps} L\mu\right)^{1/2}\ln^{1/2}\left(\frac{1}{\theta^2}\right)\geq\hat{\eps}L$, then the laminate construction $u_{Lam}$ (see (iv) above) is admissible, for which $J(u_{Lam})\lesssim\left(\hat{\eps} L\mu\right)^{1/2}\ln^{1/2}(\frac{1}{\theta^2})$ . Otherwise, we use the single laminate $u_{SL}$ (which corresponds to the choice $N=1$, $h=\theta$ and $\ell=L$ in TSB), for which  $J(u_{SL})\lesssim\theta^2(\hat{\eps} L+\mu\ln\frac{1}{\theta^2})\lesssim\theta^2\hat{\eps} L$. 
\item[(iii)] Finally, suppose that the minimum in \eqref{eq:scaling2ub} is attained by $(\hat{\eps} L\mu)^{1/2}\ln^{1/2}(3+\frac{\hat{\eps}}{\mu^3 L})$, and consider again the two possibilities separately. If $\left(\hat{\eps} L\mu\right)^{1/2}\ln^{1/2}(3+\frac{\hat{\eps}}{\mu^3 L})\geq \hat{\eps}L$, we use TSB as in (v) above, i.e., with $N\sim\frac{\mu^{1/2}\ln^{1/2}\left(3+\frac{\hat{\eps}}{\mu^3 L}\right)}{\hat{\eps}^{1/2}L^{1/2}}$, $\ell=L$ and $h=\frac{L^{1/2}\mu^{3/2}\ln^{1/2}\left(3+\frac{\hat{\eps}}{\mu^3 L}\right)}{\hat{\eps}^{1/2}}$, which yields a function $u_{TSB}$ with $J(u_{TSB})\lesssim(\hat{\eps} L\mu)^{1/2}\ln^{1/2}\left(3+\frac{\hat{\eps}}{\mu^3 L}\right)$. Note that admissibility of $h$ follows verbatim as in the case of Neumann boundary conditions. Otherwise, we have $\left(\hat{\eps} L\mu\right)^{1/2}\ln^{1/2}(3+\frac{\hat{\eps}}{\mu^3 L})\leq \hat{\eps}L$, and we use TSB with $N=1$, $\ell=\frac{\mu\ln(3+\frac{\hat{\eps}}{\mu^3 L})}{\hat{\eps}}$ and $h$ as before. 
Note that $\ell$ is chosen to make $N=1$. In the regime under consideration, we indeed have $\ell\leq L$. Then by \eqref{eq:TSBsumD}, this yields a function $u_{TSB}^{(P)}$ with $J(u_{TSB}^{(P)})\lesssim\theta^2((\hat{\eps}L\mu)^{1/2}+\hat{\eps}L+\mu\ln\frac{\hat{\eps}}{\mu^3 L})\lesssim\theta^2\hat{\eps}L$.
%Then 
%comparing to $\hat{\eps}^{2/3}L^{1/3}$ implies that $\hat\eps\geq \mu^3 L$, and comparing to $(\hat{\eps}L\mu)^{1/2}\ln^{1/2}\frac{1}{\theta}$ yields $\hat\eps\leq \mu^3 L/\theta^2$.
%Therefore the two scale-branching construction $u_{TSB}$ is admissible, and the conclusion follows.
\end{itemize}
We remark that in the case $\mathcal{A}_P$ all functions obey $u(x_1,0)=0$ for all $x_1$.
\end{proof}

\subsection{Lower bound}
  We now turn to the proof of the lower bounds of Theorem \ref{th:1}. {\BZ{Our main new contribution is the proof in the case $\theta\ll 1$, in particular Step 2 of the proof of Lemma \ref{lemma1} below. }} We proceed in two steps. 

\begin{lemma}\label{lemma1}
There exist $m_0\in(0,1/2]$ and $c>0$ with the following properties: For all $\mu,\eps>0$, 
$\theta\in (0,1/2]$, $\lambda\in (0,1]$, $\ell\in (0,L]$, $m\in(0,m_0]$ such that
\begin{equation}\label{lemma1cases}
\text{ at least one of }\hskip1cm \theta\le m\le m_0 \hskip1cm\text{ or }\hskip1cm   m_0=m\le\theta 
\end{equation}
holds,  one has with $\hat{\eps}=\eps/\theta^2$
\begin{equation}\label{eqlemma1}
J(u)\ge c\theta^2 \min\left\{ \frac{\hat{\eps} \ell
}{\lambda},\ \mu,\ \mu\lambda \ln \frac1 m ,
   \frac{ \lambda^2m}{\ell
}(\ln \frac1 m)^2 \right\}\,\text{\qquad for all\ }u\in\mathcal{A}_N,
\end{equation}
and 
\begin{equation}\label{eqlemma1.1}
J(u)\geq c\theta^2 \min\left\{\frac{\hat{\eps} \ell
}{\lambda},\ \mu\lambda \ln\frac{1}{ m},\,\frac{\lambda^2 m}{\ell}(\ln\frac{1}{ m})^2 \right\}\text{\qquad for all \ }u\in\mathcal{A}_P.
\end{equation}
\end{lemma}

\begin{proof}
\emph{Step 1. General setting and localization.}
We start with the assertion \eqref{eqlemma1}.  Let $\mu$, $\eps$, $\theta$, $\lambda$, $\ell$ and $m$ be given such that the assumptions of the lemma are satisfied (without $m_0$ which will be chosen at the end of the proof), and let $u\in\mathcal{A}_N$. For  ease of notation we 
shall work for most of the proof with the function $v:(-L,\infty)\times(0,1)\rightarrow\R$ given by $v(x_1,x_2):=-u(x_1,x_2)+\theta x_2$. 
Then $\partial_2v\in\{0,1\}$ almost everywhere in $(-L,0)\times(0,1)$ and
\begin{eqnarray}
\label{eq:KMtrans}
J(u)=\mu\int_{(0,\infty)\times(0,1)} |Dv(x)-\theta e_2|^2\dx+\int_{-L}^0\int_0^1 (\partial_1v(x))^2\dx+\eps|\partial_2\partial_2v|=:F(v).
\end{eqnarray}
For any $x_1^*\in (-\ell,0)$ we estimate
\begin{equation*}
\min_{\alpha\in\R} \|v(x_1^*,x_2)-\theta x_2-\alpha\|_{L^1((0,1))}
 \le \|v(x_1^*,x_2)-v(0,x_2)\|_{L^1((0,1))}
 +\min_{\alpha\in\R}\|v(0,x_2)-\theta x_2-\alpha\|_{L^1((0,1))}
\end{equation*}
(with the functions $v(x_1,\cdot)$ understood as traces).
The first term is controlled by
\begin{eqnarray*}
%\label{eq:eins}
\|v(x_1^*,\cdot)-v(0,\cdot)\|_{L^1((0,1))}&\leq& \ell^{1/2}\left(\int_0^1\int_{-\ell}^0\left|\partial_1v(x_1,x_2)\right|^2\dx\right)^{1/2}
\leq \ell^{1/2}(F(v))^{1/2},
\end{eqnarray*}
and the second one by
\begin{align*}%\label{eq:zwei}
\min_{\alpha\in\R}\|v(0,x_2)-\theta x_2-\alpha\|_{L^1((0,1))}
&\le \min_{\alpha\in\R}\|v(0,x_2)-\theta x_2-\alpha\|_{L^2((0,1))}\\
&\le c \|Dv-\theta e_2\|_{L^2((0,1)^2)}
\leq\frac{c}{\mu^{1/2}}(F(v))^{1/2}\,,
\end{align*}
where we used H\"older's inequality and the trace theorem.
Therefore
\begin{equation}\label{eqlbKMl101}
\min_{\alpha\in\R} \|v(x_1^*,x_2)-\theta x_2-\alpha\|_{L^1((0,1))}
 \le \ell^{1/2}(F(v))^{1/2}+
 \frac{c}{\mu^{1/2}}(F(v))^{1/2}.
\end{equation}
 %We can assume $ m\le m_0$, for some universal constant $ m_0<1$ chosen below.
%To avoid the periodicity assumption, we work vertically on 
Choose $x_1^*\in [-\ell,0]$ which minimizes $|\partial_2\partial_2v|(\{x_1\}\times
(0,1))$. Such a point $x_1^\ast$ exists since $|\partial_2\partial_2v|(\{x_1\}\times (0,1))$ is discrete
(by  \cite[Section 4]{giuliani-mueller:12} we also know that we can assume without loss of generality that 
$x_1^\ast=-\ell$, this is however not needed here).
Consider 
the central
interval $I:=(1/3,2/3)$. 
Let $I_*:=\{x_1^*\}\times I$ and  $n:=|\partial_2\partial_2v|(I_*)$.
If $n\lambda\ge 1$ then 
\begin{eqnarray}\label{eq:lemmaposs1}
F(v)\ge \frac{\ell
 \eps}{\lambda},
 \end{eqnarray} and the assertion is proven. 
Otherwise, consider 
$\omega :=\{x_2\in I: \partial_2v(x_1^*,x_2)=1\}$. Note that $\omega$ consists of at most $1/\lambda$ intervals.
We claim that we only need to treat the case $|\omega|\sim\theta$. 
Indeed, for every $\beta\in\R$,
\begin{equation*}%\label{eq:possibilities}
 c\theta \le \min_{\alpha\in\R} \|\theta x_2-\alpha\|_{L^1(I)} \le 
  \|v(x_1^*,x_2)-\theta x_2-\beta\|_{L^1(I)} +
 \min_{\alpha\in\R} \|v(x_1^*,\cdot)-\alpha\|_{L^1(I)}.
\end{equation*}
Since $\partial_2 v\in\{0,1\}$ almost everywhere, we have $\|v(x_1^*,\cdot)-\alpha\|_{L^\infty(I)}\le \frac12
|\omega |$ for some $\alpha\in\R$. 
Recalling (\ref{eqlbKMl101}) we get
\begin{equation*}%\label{eq:possibilities}
 c\theta\le  \ell^{1/2}(F(v))^{1/2}+
 \frac{c}{\mu^{1/2}}(F(v))^{1/2}+|\omega|\,,
\end{equation*}
therefore either $|\omega|\gtrsim\theta$ or at least one of
\begin{eqnarray}\label{eq:possneumann}
\label{eq:cases}
F(v)\gtrsim\frac{\theta^2}{\ell} \text{\qquad or \qquad}F(v)\gtrsim\mu\theta^2
%\text{,\qquad or \qquad}
%|\omega |\gtrsim\theta.
\end{eqnarray}
must hold. If (\ref{eq:possneumann}) holds, the assertion is established, since $\lambda^2 m\ln^2\frac{1}{m}\leq 1$ for all $\lambda,m\in(0,1)$.
  We thus can focus on the case  $|\omega|\gtrsim\theta$
and observe that if $|\omega|\ge2\theta$ then
 $v(x_1^*,x_2)\ge v(x_1^*,x_2')+2\theta$ for all $x_2\in (2/3,1)$ and all $x_2'\in (0,1/3)$, leading to (recall \eqref{eqlbKMl101})
\begin{equation*}%\label{eqlbKMl101}
c\theta \le \min_{\alpha\in\R} \|v(x_1^*,x_2)-\theta x_2-\alpha\|_{L^1((0,1))}
 \le \ell^{1/2}(F(v))^{1/2}+
 \frac{c}{\mu^{1/2}}(F(v))^{1/2}.
\end{equation*}
 Again, either one of the possibilities in (\ref{eq:possneumann}) holds and the assertion
 is proven, or $|\omega|<2\theta$. In the rest we hence consider the case ${\BZ{c\theta\leq |\omega|\leq 2}}\theta$.

\emph{Step 2: Construction of the test function and proof of \eqref{eqlemma1} for small $\theta\leq m\leq m_0$.}
The set  $B_{\lambda m}(\omega )$, which is the $\lambda m$-neighbourhood of $\omega$, is the union of at most $n<1/\lambda$ intervals $(y_i-g_i,
y_i+g_i)$, where $y_i$ are the midpoints of the connected components of $\omega$. 
%with
%$\sum g_i \le |\omega |+2n\lambda m\le 2\theta+2 m\le 4 m$ and $g_i\ge \lambda m$. 
Then $\lambda m\leq g_i\leq \lambda m+\frac{|\omega|}{2}\leq \lambda m+\theta$ for every $i$, and $\sum_{i=1}^n 2g_i\leq |\omega|+2n\lambda m\leq 2\theta+2m\leq 4m$. 
We define
\begin{equation}\label{eq:defpsi}
  \psi(t):=\max_i \psi_i(t-y_i)\,,\hskip1cm
  \psi_i(t):=\left[ \ln \frac{1}{6m} - \left(\ln \frac{|t|}{g_i}\right)_+ \right]_+\,.
\end{equation}
Then $\psi_i=\ln\frac 1 {6m}$ in $(-g_i,g_i)$, and $\psi_i=0$ outside $(-g_i/ (6m),
g_i/ (6m))$. %Further, $\supp \psi\subset\bigcup_{i=1}^n(y_i-g_i/(6m),y_i+g_i/(6m))\subset\bigcup_{i=1}^n(y_i-\lambda/6-\theta/(6m),y_i+\lambda/6+\theta/(6m))$, and 
Hence, 
since $g_i/(6m)\leq (\lambda m+\theta)/(6m)\le1/3$ and $y_i\in\omega\subset(1/3,2/3)$, we deduce that $\supp\psi\subset[0,1]$.
%if $ m_0\le 1/3$ then $\psi(t)=0$ for $t\in\{0,1\}$.
Note that $\|\psi_i\|_{L^1(\R)}\le c g_i/ m$ and therefore 
\begin{eqnarray}\label{eq:psiL1}
  \|\psi\|_{L^1((0,1))} \le \sum_{i=1}^n \|\psi_i\|_{L^1(\R)}\leq\frac{c}{m}\sum_{i=1}^n g_i\leq c.
\end{eqnarray}
Further,
\begin{eqnarray}\label{eq:psi'L2}
\|\psi'\|_{L^2((0,1))}\le \left(\sum_{i=1}^n\int_0^1|\psi_i'(t)|^2\dt\right)^{1/2}\leq c\left(\sum_{i=1}^n\frac{1}{g_i}\right)^{1/2}\leq c\left(\frac{n}{\lambda m}\right)^{1/2}
\le  \frac{c}{m^{1/2}\lambda}.
\end{eqnarray}
We next estimate $[\psi]_{H^{1/2}_P((0,1))}$. For every $i=1,\dots n$ let $\Psi_i$ be the radially symmetric extension of $\psi_i$ to $\R^2$. Then 
\begin{eqnarray*}
\int_0^\infty\int_{-1/3}^{1/3}|D\Psi_i(x)|^2\dx\sim\int_{g_i}^{g_i/(6m)} r\frac{1}{r^2}\,\dr\sim \ln\frac{1}{m},
\end{eqnarray*}
and hence,
\begin{eqnarray}\label{eq:psiH1/2P}
[\psi]_{H^{1/2}_P((0,1))}^2&\leq&
%\int_0^\infty\int_0^1|D\max\Psi(x_1,x_2-y_i)|^2\,dx\leq
\sum_{i=1}^n\int_0^\infty\int_0^1|D\Psi_i(x_1,x_2-y_i)|^2\,\dx
\lesssim n\ln\frac{1}{m}\leq\frac{1}{\lambda}\ln\frac{1}{m}.
\end{eqnarray}
% seminorm of $\psi_i$ by the $H^1$ seminorm of its
%radially symmetric extension to $\R^2$,
%\begin{equation}
%  [\psi_i]_{H^{1/2}(\R)}^2\le
%  \|D\psi_i\|_{L^2(\R^2)}^2=\int_{g_i}^{g_i/ m} 2\pi r \frac{1}{r^2} \dr
%   = 2\pi \ln \frac1 m\,.
%\end{equation}
%Analogously,
%\begin{equation}
%  [\psi]_{H^{1/2}(\R)}^2\le
%  \|D\psi\|_{L^2(\R^2)}^2\le \sum_i   \|D\psi_i\|_{L^2(\R^2)}^2\le  \frac{c}{\lambda} \ln \frac1 m\,. 
%\end{equation}
%At the same time 
%{\emph{Step 3: Conclusion of the proof of \eqref{eqlemma1} for small $\theta$.}} 
Since $\partial_2v(x_1^\ast,\cdot) \psi= \ln \frac{1}{6m}$ on $\omega$, and
$\partial_2v(x_1^\ast,\cdot)\psi\ge 0$ everywhere, we have by \eqref{eq:interpolFourier2} and with $u_R(x_2):=\theta x_2$,
\begin{eqnarray}\label{eq:compN}
&& |\omega|  \ln \frac1{6 m} \le \int_{0}^1 \partial_2v(x_1^\ast,x_2) \psi(x_2)\dy\nonumber\\
&=& -\int_{0}^1(v(x_1^\ast,x_2)-v_0(x_2))\psi'(x_2)\dy-\int_0^1(v_0(x_2)-\theta x_2)\psi'(x_2)\dy+\int_0^1\theta \psi(x_2)\dy
\nonumber\\
&\leq&\int_{x_1^\ast}^0\int_0^1\partial_1v(x_1,x_2)\psi'(x_2)\dx+c[v_0-u_R]_{H_N^{1/2}((0,1))}[\psi]_{H^{1/2}_P((0,1))}+\theta\|\psi\|_{L^1((0,1))}\nonumber\\
&\leq&\frac{c}{m^{1/2}\lambda} |x_1^\ast|^{1/2}F^{1/2}(v)+\frac{cF^{1/2}(v)}{(\lambda\mu)^{1/2}}\ln^{1/2}\frac{1}{m}+c\theta,
\end{eqnarray}
where in the last step we used  \eqref{eq:psiL1}, \eqref{eq:psi'L2} and \eqref{eq:psiH1/2P}.
%In the last term we integrated by parts,
%\[\int_{(0,1)} (u_0'(x_2)-\partial_2u(-\ell,x_2))\psi \dy=-\int_{(0,1)} (u_0(x_2)-u(-\ell,x_2))\psi' \,dx_2 = -\int_{(0,-\ell)\times (0,1)} \partial_1u \psi'\,dx_1\,dx_2,\] and used H\"older's inequality.
%With $[u_0'-\theta]_{H^{-1/2}}=[u_0-u_R]_{H^{1/2}}$ 
%and ,
%\begin{alignat}1
%  c_1 \ln\frac1 m \theta \le
%(F(v)/\mu)^{1/2}  [\psi]_{H^{1/2}}
%+ \theta  \|\psi\|_{L^1}
%+ c\,\ell
%^{1/2}  \frac{ (F(v))^{1/2} }{\lambda m^{1/2}}\,.
%\end{alignat}
%Since $\|\psi\|_1\le c$,
Since $|\omega|\gtrsim\theta$, 
there is $ m_0\in(0,1/2]$ such that for  $ m\le m_0$ the last term of the right-hand side can be absorbed in the left-hand side {\BZ{and $\ln\frac{1}{6m}\gtrsim\ln\frac{1}{m}$}} 
(this defines $ m_0$).
We deduce that 
%\begin{alignat}1
%  c_1 \ln\frac1 m \theta \le
%(F(v)/\mu)^{1/2} 
%\frac{ (\ln (1/ m))^{1/2}}{\lambda^{1/2}}
%+ \frac{c \ell
%^{1/2}  (F(v))^{1/2} }{\lambda m^{1/2}}
%\end{alignat}
%or, equivalently,
\begin{equation}
(F(v))^{1/2} \ge  c\theta \min\left\{\frac{\lambda m^{1/2} }{\ell
^{1/2} } \ln \frac1 m,\ \lambda^{1/2} \mu^{1/2} (\ln\frac1 m)^{1/2}
\right\},
\end{equation}
which, combined with \eqref{eq:lemmaposs1} and  \eqref{eq:possneumann}, concludes the proof of \eqref{eqlemma1}
if the first option in (\ref{lemma1cases}) holds.

{\emph{Step 3: Construction of the test function and proof of \eqref{eqlemma1} for large $\theta\geq m_0=m$.}} 
Note that 
if the second option in (\ref{lemma1cases}) holds
then $m\sim \theta\sim 1$. {\BZ{In this case, the argument from \cite{conti:06} can be applied to obtain the lower bound but we present here an alternative proof in the spirit of {\em Step 2}.}} 
At variance with what was done in Step 2, in this case the test function is supported inside $\omega$.
The set $\omega$ is the union of $n\le 1/\lambda$ intervals, 
$(y_i-g_i,y_i+g_i)$, and $|\omega|=\sum_i 2 g_i\sim 1$. {\BZ{Note that the points $y_i$ are as in {\em Step 2} chosen to be the midpoints of the connected components of $\omega$, but the radii $g_i$ are different.}} We set
\begin{equation*}
  \psi(t):= \max_i ( g_i-|t-y_i|)_+=\begin{cases}
            \dist(t,\partial\omega) & \text{ if } t\in\omega,\\
	      0 & \text{ otherwise.}
           \end{cases}
\end{equation*}
We compute
\begin{equation*}
 \int_0^1\psi(t)\dt = \sum_{i=1}^n 2 \int_0^{g_i} s\, \ds = \sum_{i=1}^n g_i^2 \ge \frac1n (\sum_{i=1}^n g_i)^2 \gtrsim \lambda
\end{equation*}
and
\begin{equation*}
 \int_0^1|\psi'(t)|^2\dt = \sum_{i=1}^n 2g_i =|\omega| \le 1\,.
\end{equation*}
Extending $\psi$ to $(0,\infty)\times(0,1)$ by $\Psi(x):=\max_i ( g_i-|x-(0,y_i)|)_+$, we obtain
\begin{equation*}
 [\psi]_{H^{1/2}_P((0,1))}^2\le \int_{(0,\infty)\times(0,1)}|D\Psi(x)|^2\dx 
  = \sum_{i=1}^n \frac12 |B(0, g_i)| =\frac\pi2 \sum_{i=1}^n g_i^2\,.
\end{equation*}
As in (\ref{eq:compN}), we use $\psi$ to test $\partial_2 v$, integrate by parts and use (\ref{eq:interpolFourier2}),
\begin{align*}
\sum_{i=1}^n g_i^2 =& \int_{0}^1  \psi(x_2)\dy=
 \int_{0}^1 \partial_2v(x_1^\ast,x_2) \psi(x_2)\dy\nonumber\\
=& -\int_{0}^1(v(x_1^\ast,x_2)-v_0(x_2))\psi'(x_2)\dy-\int_0^1(v_0(x_2)-\theta x_2)\psi'(x_2)\dy+\int_0^1\theta \psi(x_2)\dy
\nonumber\\
\leq&\int_{x_1^\ast}^0\int_0^1\partial_1v(x_1,x_2)\psi'(x_2)\dx+c[v_0-u_R]_{H_N^{1/2}((0,1))}[\psi]_{H^{1/2}_P((0,1))}+\theta
\sum_{{\BZ{i=1}}}^{\BZ{n}} g_i^2\,.
\end{align*}
Since $\theta\le 1/2$, 
\begin{align*}
\sum_{i=1}^n g_i^2
\leq&c |x_1^\ast|^{1/2}F^{1/2}(v)+\frac{cF^{1/2}(v)}{\mu^{1/2}} \left({\com{\sum_{i=1}^n}} g_i^2\right)^{1/2}\,.
\end{align*}
Therefore at least one of 
\begin{equation*}
 F(v)\gtrsim \frac{1}{\ell} \left(\sum_{i=1}^n g_i^2\right)^{2}\gtrsim \frac{\lambda^2}{\ell}
\hskip5mm\text{and}\hskip5mm
 F(v)\gtrsim \mu \sum_{i=1}^n g_i^2\gtrsim \mu\lambda
\end{equation*}
must hold and the proof of \eqref{eqlemma1} is concluded.

{\emph{Step 4: Proof of \eqref{eqlemma1.1}}}
Let $u\in\mathcal{A}_P$ and choose $v$ and $x_1^\ast$ as in {\emph{Step 1}}. By the periodicity assumption, we have %$J(u)\geq\eps L$ and 
$|\{x_2\in(0,1):\partial v(x_1^\ast,x_2)=1\}|=\theta$. Hence, possibly choosing a different cell of periodicity, we may assume that $\theta/3\leq|\omega|\leq\theta$, where $\omega$ is defined as in {\emph{Step 1}}, i.e., $\omega:=\{x_2\in(1/3,2/3):\partial_2v(x_1^\ast,x_2)=1\}$. 
The rest of the proof is the same as for $\mathcal{A}_N$.
We remark that deriving the condition $|\omega|\sim\theta$ directly from the periodicity allowed us to avoid using (\ref{eq:possneumann}), and therefore 
the option $F(v)\gtrsim \mu\theta^2$ does not appear in the conclusion.
\end{proof}

We now conclude the proof of the lower bounds of Theorem \ref{th:1}.
\begin{proposition} \label{prop:KMlb}There is a constant $c>0$ such that for all $\mu$, $\eps$, $L>0$ and all $\theta\in(0,1/2]$, we have
with $\hat\eps=\eps/\theta^2$,
  \begin{equation}\label{eq:lb1}
    \min_{u\in\mathcal{A}_N}J(u)\ge c \theta^2 \min\left\{ \hat\eps^{2/3}L^{1/3},
( \hat\eps \mu L)^{1/2}\ln^{1/2} (3+\frac{\hat\eps}{\mu^3L  }), 
 (\hat\eps \mu L)^{1/2} \ln^{1/2} \frac1\theta, \mu, \hat\eps^{1/2} \right\}\,,
  \end{equation}
  and
  \begin{equation}\label{eq:lb2}
  \min_{u\in\mathcal{A}_P}J(u)\geq c \theta^2\max\left\{\hat{\eps} L,\ \min\left\{\hat{\eps}^{2/3}L^{1/3},\ (\hat{\eps} L\mu)^{1/2}\ln^{1/2}(3+\frac{\hat{\eps}}{\mu^3 L}),\,\left(\hat{\eps} L\mu\right)^{1/2}\ln^{1/2}\frac{1}{\theta}\right\}\right\}.
  \end{equation}
  \end{proposition}
\begin{proof}
To prove \eqref{eq:lb1}, we use \eqref{eqlemma1} with appropriate choices for the parameters $\lambda$, $ m$ and $\ell$. 
We start fixing $u\in \mathcal A_N$, and distinguish several cases.
\begin{itemize}
\item[(i)] If $\hat{\eps}\leq\min\{1/L^2,\ \mu^3L\}$, we choose $ m:=m_0$, $\lambda:=\hat{\eps}^{1/3}L^{2/3}$, and $\ell
:=L$. 
Then (\ref{lemma1cases}) holds for all $\theta$, and by \eqref{eqlemma1} we obtain $J(u)\gtrsim\theta^2\min\{\hat{\eps}^{2/3}L^{1/3},\ \mu\}$. 
\item[(ii)] If $1/L^2\leq\hat{\eps}\leq\mu^3 L$, we choose $ m:=m_0$, $\lambda:=1$ and $\ell
:=\hat{\eps}^{-1/2}$. 
As above,  \eqref{eqlemma1} implies
$J(u)\gtrsim\theta^2\min\{\mu,\ \hat{\eps}^{1/2}\}$.
\item[(iii)] It remains to consider $\hat{\eps}>\mu^3L$.
We set
\begin{equation}\label{eqdefm}
m:=\max\{\min\{m_0,\theta\}, \frac{\mu^{3/2}L^{1/2}}{\hat{\eps}^{1/2}}m_0\}\,.
\end{equation}
If $\theta\ge m_0$ this gives $m=m_0$; if $\theta\le m_0$ instead $\theta\le m\le m_0$. Therefore we can use  \eqref{eqlemma1}. Note that $\ln\frac{1}{m}\geq\min\{\ln\frac{1}{\theta},\ \ln\frac{\hat{\eps}^{1/2}m_0}{\mu^{3/2}L^{1/2}}\}\gtrsim\min\{\ln\frac{1}{\theta^2},\ \ln(3+\frac{\hat{\eps}}{\mu^3 L})\}$.
We distinguish two subcases:
\begin{itemize}
\item[(a)] If $\hat{\eps}\le\frac{\mu}{L}\ln\frac{1}{ m}$, we choose $\ell
:=L$ and $\lambda:=\left(\frac{\hat{\eps}L}{\mu\ln\frac{1}{ m}}\right)^{1/2}$. Then (since $ m\gtrsim\frac{\mu^{3/2}L^{1/2}}{\hat{\eps}^{1/2}}$ and $\ln\frac{1}{ m}\gtrsim1$), we have $J(u)\gtrsim\theta^2\min\{(\hat{\eps}L\mu)^{1/2}\ln^{1/2}\frac{1}{ m},\ \mu\}$.
\item[(b)] If $\hat{\eps}>\frac{\mu}{L}\ln\frac{1}{ m}$, we choose $\lambda:=1$ and $\ell
:=\frac{\mu^{3/4}L^{1/4}}{\hat{\eps}^{3/4}}$. Then $J(u)\gtrsim\mu\theta^2$.
\end{itemize}
\end{itemize}
We finally prove \eqref{eq:lb2} using \eqref{eqlemma1.1}. We fix $u\in \mathcal A_P$. By the periodicity assumption we have $J(u)\gtrsim \theta^2\hat\eps L$.
 As above, we distinguish several cases.
\begin{itemize}
\item[(i)] 
If $\hat{\eps}\leq\min\{1/L^2,\ \mu^3L\}$ we choose $ m:=m_0$, $\lambda:=\hat{\eps}^{1/3}L^{2/3}$, and $\ell
:=L$ as in (i) above and obtain  $J(u)\gtrsim\theta^2\hat{\eps}^{2/3}L^{1/3}$.
\item[(ii)] If $\hat\eps\ge L^{-2}$ then
$J(u)\gtrsim \theta^2\hat\eps L\gtrsim \theta^2(\hat{\eps} L +\hat\eps^{2/3}L^{1/3})$ and the assertion is proven.
\item[(iii)] It remains to consider $\hat{\eps}>\mu^3L$.
We define $m$ as in (\ref{eqdefm}) and distinguish the same two subcases as above.
\begin{itemize}
\item[(a)] If $\hat{\eps}\le\frac{\mu}{L}\ln\frac{1}{ m}$, we choose the same parameters as in (iii)(a) above and obtain $J(u)\gtrsim\theta^2(\hat{\eps}L\mu)^{1/2}\ln^{1/2}\frac{1}{m}$.
\item[(b)] If $\hat{\eps}>\frac{\mu}{L}\ln\frac{1}{ m}$, 
then as in (ii) we obtain
$J(u)\gtrsim \theta^2\hat\eps L\gtrsim \theta^2 (\hat\eps \mu L)^{1/2} \ln^{1/2} \frac1m$.
\end{itemize}
\end{itemize}
\end{proof}
\section{Proof of the scaling law for plastic microstructures}\label{sec:disloc}
We now turn to the model for plastic microstructures described in Section \ref{sec:dislocmodel}, and aim to prove Theorem \ref{th:2}. 
We note that it suffices to prove Theorem \ref{th:2} for $L=1$. 
The result for general cubes $\Omega:=(0,L)^3$ then follows by rescaling $\tilde{u}(x):=\frac{1}{L}u(Lx)$ and $\tilde{\beta}(x):=\beta(Lx)$ (see \cite[Section 4]{conti-ortiz:05}).
\subsection{Upper bound}
We start with the upper bound.
\begin{proposition}
\label{prop:dislocub}
There is a constant $c>0$ such that for all ${\com{\theta\in(0,\frac{1}{2}]}}$, and all $\mu$, $\eps>0$, we have
\begin{eqnarray}\label{eq:disub}
 \inf E(u,\beta)\leq c\theta^2\min\left\{\hat{\eps}^{2/3},\,(\hat{\eps} \mu)^{1/2}\ln^{1/2}\left(3+\frac{\hat{\eps}}{\mu^3 }\right),\,\left(\hat{\eps} \mu\right)^{1/2}\ln^{1/2}\left(\frac{1}{\theta^2}\right),\, \mu,\,1\right\}.
\end{eqnarray}
\end{proposition}
\begin{proof}
\def\vKM{v_\mathrm{KM}}
Following \cite[Section 4.1]{conti-ortiz:05}, test functions for the Kohn-M\"uller model with appropriate modifications can be used to construct test functions for the energy {\com{given in \eqref{eq:disloc}}}. We treat the regimes $\mu\theta^2$  and $\theta^2$ separately, and afterwards outline the general construction and then consider the various regimes. Set 
\[f(\eps,\mu,\theta):=\theta^2\min\left\{\hat{\eps}^{2/3},\,(\hat{\eps} \mu)^{1/2}\ln^{1/2}\left(3+\frac{\hat{\eps}}{\mu^3 }\right),\,\left(\hat{\eps} \mu\right)^{1/2}\ln^{1/2}\left(\frac{1}{\theta^2}\right),\, \mu,\,1\right\}.
 \]
 {\em Step 0: The regimes $\mu\theta^2$ and $\theta^2$. } Choosing $u(x):=(1-\theta)x_\xi
+\theta x_\eta \min\{1,\dist(x,\Omega)\}$ and $\beta:=(1-\theta) e_\xi$, we find that $\inf E(u,\beta)\lesssim\mu\theta^2$.
Choosing $u:=u_0$ and $\beta:=e_\xi$ we get $\inf E(u,\beta)\le 2\theta^2$.

{\em Step 1: The general setting.}  Suppose that there are $T\leq 0$ and a function 
%In particular, for any $\eps,\theta$ there is a function
$\vKM:(-1,\infty)\times(0,1)\to\R$ such that 
\begin{eqnarray}\label{eq:vKM1}
\partial_2 \vKM(x)\in\{0,1\}\text{\quad for\ } x_1\leq T,\quad  
\vKM(x_1,0)=0, \quad \vKM(x_1,1)=\theta \text{\quad for all\ }x_1\in(-1,\infty),
\end{eqnarray}
and 
\begin{align}\label{eq:vKM2}
& \int_{(-1,T)\times(0,1)} (|\partial_1\vKM(x)|^2\dx + \eps |D\partial_2\vKM|)\\
& +\int_{(T,0)\times(0,1)}|D\vKM(x)|^2\dx+\int_{(0,\infty)\times(0,1)} |D\vKM(x)-\theta e_2|^2\dx \le c f(\eps,\theta,\mu).
\nonumber
\end{align}
%Further, 
%\begin{equation}
% \int_{(-1,H)\times(0,1)} \eps |D\partial_2\vKM|\le c f(\eps,\theta,\mu)
%\end{equation}
We use the signed distance
\begin{equation}
 d(x):=\begin{cases}
	-\dist((x_1,x_3),\partial(0,1)^2) &\text { if } (x_1,x_3)\in(0,1)^2,\\
	  \dist((x_1,x_3),(0,1)^2) &\text { if } (x_1,x_3)\not\in(0,1)^2
      \end{cases}
\end{equation}
to define
\begin{equation}
 u(x):=\begin{cases}
  x_\xi-\sqrt2\vKM(d(x),x_2) &\text{ if } x_2\in (0,1),\\
       u_0(x)& \text{ if } x_2\not\in (0,1)
      \end{cases}
\end{equation}
and 
\begin{equation}
\beta(x):=
 e_\xi-(\sqrt2\partial_2\vKM(d(x),x_2) ) \chi_{\{d(x)\le T\}}e_2\,.
\end{equation}
We collect some properties of these functions: 
Since $u_0(x)=x_\xi-\sqrt2\theta x_2$, the function $u$ is continuous at $x_2\in\{0,1\}$. For the rest we only need to consider the stripe $x_2\in(0,1)$.
Further, $\beta\in\{e_\xi,e_\eta\}$,
\[|Du-\beta|(x)=\begin{cases}
\sqrt2 |\partial_1\vKM(d(x),x_2)|&\text{\ if\ }d(x)\leq T,\\
\sqrt{2}|D\vKM(d(x),x_2)|&\text{\ if\ }T\leq d(x)\leq 0,
\end{cases} \]
 %and 
%$|Du-\beta|=\sqrt2 |\partial_1\vKM(d(x),x_2)|$ almost everywhere in $\Omega$,
$|Du-Du_0|=\sqrt2|D\vKM(d(x),x_2)-\theta e_2|$ outside $\Omega$. 
Finally, $|D\beta|\le \sqrt{2} |D\partial_2\vKM|(d(x),x_2)$ on $\{d(x)< T\}$,
 $|D\beta|=0$ on $\{d(x)>T\}$, 
 and $|D\beta|(\{d(x)=T\}) \lesssim \theta$.
 Therefore, by \eqref{eq:vKM2} (for the ease of notation we suppress the constraint $x_2\in(0,1)$ in all domains of integration)
\begin{align*}
E(u,\beta)&\lesssim\int_{\{d(x)\leq T\}}|\partial_1\vKM(d(x),x_2)|^2\dx+\int_{\{T\leq d(x)\leq 0\}}|D\vKM(d(x),x_2)|^2\dx+\\
&+\int_{\{d(x)\leq T\}}\eps|D\partial_2\vKM(d(x),x_2)|+\mu\int_{\{d(x)>0\}}|D\vKM-\theta e_2|^2\dx
+\theta^3\hat\eps
\lesssim f(\eps,\theta,\mu){\SC +\theta^3\hat\eps.}
\end{align*}
{\em Step 2: Construction of $\vKM$. }It remains to find $\vKM$ and $T$ such that \eqref{eq:vKM1} and \eqref{eq:vKM2} hold. We consider the various regimes separately and provide functions $v_{KM}$ and parameters $T\leq 0$ that satisfy \eqref{eq:vKM1} and \eqref{eq:vKM2}. For that, we build on the functions from Section \ref{secpfKM} and in particular the proof of Proposition \ref{prop:KMub} with $L=\ell=1$.
{\SC In each case it is easy to see that the $\theta^3\hat \eps$ term is irrelevant.}
\begin{itemize}
%\item[(i)] If $f(\eps,\theta,\mu)=\mu\theta^2$, the choose $\vKM(x)=-u_{SL}(x)+\theta x_2$ and $T=0$. 
%\item[(i)] If $f(\eps,\theta,\mu)=\theta^2$, then choose $T=-1$ and $\vKM(x)=\theta x_2$. (Note that in the martensite case, this corresponds to choosing $u\equiv 0$, which was ruled out by the constraint on $\partial_2u$.)
\item[(i)] If $f(\eps,\theta,\mu)=\theta^2\left(\hat{\eps} \mu\right)^{1/2}\ln^{1/ 2}\left(\frac{1}{\theta^2}\right)$ we choose $\vKM(x):=-u_{Lam}(x)+\theta x_2$ and $T:=0$.
\item[(ii)] If $f(\eps,\theta,\mu)=\theta^2\hat{\eps}^{2/3}$, we use a modification of the classical branching construction $u_B$, which is a variant of the branching construction given in \cite[Section 4.1]{conti-ortiz:05} with unequal volume fractions and corresponds to the finite branching construction given in \cite[Pf. of Theorem 2]{zwicknagl:14}. We use the notation of Section \ref{secpfKM} and set $\ell:=1$, $N\sim\hat{\eps}^{-1/3}$, $h:=1$ and $T:=-3^{-I}$ with $I$ given in Remark \ref{rem:trunc}. Note that for $\ell=1$, we always find an admissible $I\in\N$ since $N\geq 1$ and $\theta\leq 1/2$. 
%Note that $I\geq 1$ since $\hat{\eps}^{2/3}\leq 1$ implies that $(3/2)^I\sim N/\theta\gtrsim 1$. 
We define $\vKM$ via
\[\vKM(x):=\begin{cases}
-u_B(x)+\theta x_2&\text{\ if\ }x_1\leq T,\\
\theta x_2-\frac{x_1}{T}u_B(T,x_2)&\text{\ if\ }T\leq x_1\leq 0,\\
\theta x_2&\text{\ if\ }x_1\geq 0.
\end{cases} \]
By the computations of Section \ref{secpfKM}, in particular \eqref{eq:surface}, it remains to show 
\begin{eqnarray*}
\int_{(T,0)\times (0,1)}|D\vKM|^2\dx=\int_{(T,0)\times (0,1)}\big|\frac{u_B(T,x_2)}{T}\big|^2+\big|\theta-\frac{x_1}{T}\partial_2u_B(T,x_2)\big|^2\dx\lesssim\theta^2\hat{\eps}^{2/3}.
\end{eqnarray*}
First, since $|u_B(T,x_2)|\lesssim\theta/(2^IN)$, we have
\[ \int_{(T,0)\times (0,1)}\big|\frac{u_B(T,x_2)}{T}\big|^2\lesssim|T|\left(\frac{\theta}{2^INT}\right)^2\dx\leq\frac{\theta^2}{N^2}\sim\theta^2\hat{\eps}^{2/3}.\]
Second, since $|\partial_2u_B(T,x_2)|\leq 1-\theta$,
\begin{eqnarray*}
\int_{(T,0)\times (0,1)}\big|\theta-\frac{x_1}{T}\partial_2u_B(T,x_2)\big|^2\dx\leq |T|=3^{-I}\leq \left(\frac{2}{3}\right)^{2I}\sim\theta^2\hat{\eps}^{2/3},
\end{eqnarray*}
using $1/3\leq 4/9$ and \eqref{eq:abbruch}.
\item[(iii)] Consider now the regime in which $f(\eps,\theta,\mu)=(\hat{\eps} \mu)^{1/2}\ln^{1/2}\left(3+\frac{\hat{\eps}}{\mu^3 }\right)$. Set $\ell=1$, $N\sim(\frac{\mu\ln(3+\frac{\hat{\eps}}{\mu^3 })}{\hat{\eps}})^{1/2}$ and $h\sim N\mu$. Admissibility of these choices follows as in the case of Neumann boundary conditions in the proof of Proposition \ref{prop:KMub}. Choose $I\in\N$ such that \eqref{eq:abbruch} holds and set $T:=3^{-I}$. We define
\begin{eqnarray*}
v_{KM}(x):=\begin{cases}
-u_{TSB}(x)+\theta x_2&\text{\ if\ }x_1\notin(T,0),\\
-\frac{x_1}{T}u_{TSB}(T,x_2)-(1-\frac{x_1}{T})u_{TSB}(0,x_2)+\theta x_2&\text{\ if\ }x_1\in(T,0).
\end{cases}
\end{eqnarray*}
Again, in view of Section \ref{secpfKM}, it suffices to consider $\int_{(T,0)\times(0,1)}|D\vKM|^2\dx$. Using that $|\{x_2\in(0,1):u_{TSB}(T,x_2)\neq u_{TSB}(0,x_2)\}|\leq h$ and $|\partial_2u_{TSB}| \leq 1$, we have
\begin{align*}
&\int_{(T,0)\times (0,1)}|D\vKM|^2\dx=\\
=&\int_{(T,0)\times (0,1)}\left\{\big|\frac{u_{TSB}(T,x_2)-u_{TSB}(0,x_2)}{T}\big|^2+\big|\theta-\frac{x_1}{T}\partial_2u_{TSB}(T,x_2)-(1-\frac{x_1}{T})\partial_2u_{TSB}(0,x_2)\big|^2\right\}\dx\\
\leq&\frac{|T|h}{T^2}\left(\frac{\theta}{2^IN}\right)^2+|T|h\leq \left(\frac{3}{4}\right)^I\mu\frac{\theta^2}{N}+N\mu 3^{-I}\leq \mu\frac{\theta^2}{N}+\left(\frac{2}{3}\right)^{2I}N\mu\leq 2\mu\frac{\theta^2}{N}\lesssim(\mu\hat{\eps})^{1/2}\theta^2,
\end{align*}
 where we again used that $1/3\leq 4/9$ and \eqref{eq:abbruch}.
\end{itemize}
\end{proof}
\subsection{Lower bound}
We split the proof of the lower bound in several lemmas. We start with the local structure of $\beta$.
%The first one is a quantitative version of \cite[Lemma 4.3]{ContiOrtiz05}.
\begin{lemma}\label{lemmarect}
 Let $a,b>0$, $0<\theta\leq1/2$, and $N{\BZ{\in\N}}$, and set $R:=(0,a)e_\xi\times(0,b)e_\eta$. Suppose that $v$, $w\in L^1(R)$ with $vw=0$ almost everywhere. Assume further that
 \begin{eqnarray}\label{eq:disass1}
 \int_R v(x)\dx \ge \frac{1}{4}\theta |R| \,,%\text{\qquad and\quad}\  \int_R u(x_\xi,x_\eta)\,\dx \ge \frac{1}{4}|R|,
 \end{eqnarray}
and also
 \begin{eqnarray}\label{eq:disass2}
  \int_R |\partial_\xi v|  \le \frac{1}{16}\theta b \,,\hskip5mm\text{and\qquad}
  \int_R |\partial_\eta w| \le a N \,.
  \end{eqnarray}
Then there are $8N$ pairwise disjoint, possibly degenerate intervals $I_1,\dots, I_{8N}\subset(0,b)$ such that
\begin{equation}\label{eq:lemrect}
 \sum_{i=1}^{8N} \int_{0}^{a}\int_{I_i} v(se_\xi+te_\eta) \ds\dt \ge \frac{1}{16}\theta |R|\,,\hskip5mm\text{and\qquad}
 \sum_{i=1}^{8N} \int_{0}^{a}\int_{I_i}w(se_\xi+te_\eta) \ds\dt\ \le \frac a4 \sum_{i=1}^{8N} |I_i|\,.
\end{equation}
%In particular, one necessarily has $N\ge 1$.
\end{lemma}
\begin{proof}
For ease of notation we change coordinates for the proof, and assume $e_\xi=e_1$, $e_\eta=e_2$.
 By \eqref{eq:disass1} and Fubini, there exists $x_\xi^\ast\in (0,a)$ such that
  \begin{equation}\label{eq:65}
   \int_0^b v(x_\xi^\ast, x_\eta) \deta\ge \frac14\theta b\,.
  \end{equation}
 Define $M:=\{x_\eta\in (0,b): v(x_\xi^\ast, x_\eta) > 2 \int_{{\BZ{(0,a)\times\{x_\eta\}}}} |\partial_\xi v|\}$.
Then by \eqref{eq:disass2}
\begin{equation*}
 \int_{(0,b)\setminus M} v(x_\xi^\ast, x_\eta) \deta
 \le 2 \int_R |\partial_\xi v|\le \frac18\theta b,
\end{equation*}
and therefore by \eqref{eq:65}
\begin{equation}\label{eq:disbeta}
 \int_{M} v(x_\xi^\ast, x_\eta) \deta \ge\frac18\theta b\,.
\end{equation}
Further, for $x_\eta\in M$ one has
\begin{equation}\label{eq:disbeta2}
 v(x_\xi,x_\eta) \ge v(x_\xi^*,x_\eta)-\int_{{\SC{(0,a)\times\{x_\eta\}}}}|\partial_\xi v|  \geq \frac{1}{2}v(x_\xi^\ast,x_\eta)>0\text{\qquad for almost every $x_\xi$,}
\end{equation}
 and therefore by \eqref{eq:disbeta}
\begin{equation*}
 \int_{(0,a)\times M} v(x_\xi,x_\eta) \dx \ge
 \frac12 a \int_M v(x_\xi^\ast,x_\eta)\deta\ge 
 \frac{1}{16}\theta ab\,.
\end{equation*}
Set $\overline{w}(x_\eta):=\int_0^a w (x_\xi,x_\eta) \dxi$.
Then by the assumption $wv=0$ and \eqref{eq:disbeta2}, we have $\overline{w}=0$ in $M$, and by  \eqref{eq:disass2} %we have
\begin{equation*}
% \int_0^b w \deta \ge\frac14 ab \,,\hskip1cm \\text{and\quad}
\int_0^b|\overline{w}'| \le aN\,.
\end{equation*}
By the coarea formula, 
$\int_{a/8}^{a/4} \calH^0(\partial\{\overline{w}\le t\}) dt\le aN$, and therefore
there is $t\in (a/8,a/4)$ such that the set
$\{\overline{w}<t\}$ is the union of at most $8N$ intervals $I_1,\dots, I_{8N}$.
For these intervals \eqref{eq:lemrect} holds since they cover $M$ and $\overline{w}\leq a/4$ on all $I_i$.
\end{proof}

 For $R\subset (0,1)^2$ {\BZ{and $(u,\beta)\in\mathcal{A}$, we set}}
% we denote by $E_R$ the energy restricted to $R\times\R$, i.e.,
 \[E_R(u,\beta):=\int_{R\times(0,1)} \left[|Du(x)-\beta(x)|^2 \dx+ \eps |\partial_\xi\beta_\eta|+\eps |\partial_\eta\beta_\xi| \right]
 +\mu \int_{R\times {\BZ{[}}(-\infty,0)\cup(1,\infty){\BZ{]}}} |Du(x)-Du_0(x)|^2\dx. \]
We fix a function $\varphi_1\in C_c^1((0,1);[0,1])$ such that $\varphi_1\equiv 1$ on $(\frac{1}{16},\frac{15}{16})$ and $\varphi_1(1-t)=\varphi_1(t)$. For $a>0$ and $x\in(0,1)$ set $\varphi_{x,a}(t):=\varphi_1((t-x)/a)$. 
 \begin{lemma}\label{lemmatest}
 There exists a constant $c>0$ with the following property: For %every $\ell\in (0,1]$, 
 every $(u,\beta)\in\mathcal{A}$, 
 and every
 $a$, $b$, $x_a$, $x_b\in(0,1)$ for which $R:=(x_a,x_a+a)e_\xi\times (x_b,x_b+b)e_\eta \subset(0,1)^2$
 there exists ${\cm{\ell}}\in(0,{\cm 1})$ such that for all $f\in L^2((x_b,x_b+b))$ and all $g\in L^2((x_a,x_a+a))$,
 setting $R_{{\cm \ell}}=R\times\{{\cm{\ell}}\}$,
{\cm \begin{multline}\label{eq:choice1}
\left|\int_{R_{{\cm{\ell}}}} (\beta_\xi(x)-(1-\theta)) f(x_\eta)\varphi_{x_a,a}(x_\xi) \dx\right| 
\leq c  E_R^{1/2}(u,\beta)\| f\|_{L^2((x_b,x_b+b))} \left[
  \frac{1}{\mu^{1/2}}+ \frac{1}{a^{1/2}} 
%+(\frac{a}\ell)^{1/2}
\right]\,%\nonumber\\
%\qquad
 \end{multline}}
 and 
{\cm \begin{eqnarray}\label{eq:choice2}
\!\!\!\!\!\!\left|\int_{R_{{\cm{\ell}}}} (\beta_\eta(x)-\theta) g(x_\xi) \varphi_{x_b,b}(x_\eta)\dx\right| \le 
 c E_R^{1/2}(u,\beta)\| g\|_{L^2((x_a,x_a+a))} \left[ \frac{1}{\mu^{1/2}}
 + \frac{1}{b^{1/2}}
 %+(\frac{b}\ell)^{1/2}
 \right].   \end{eqnarray}}
\end{lemma}
\begin{proof}
Let $(u,\beta)\in\mathcal{A}$. 
Choose ${\cm{\ell}}\in(0,{\cm 1})$ such that {\cm $E_{R,{{\cm{\ell}}}}\leq E_R$}, where 
\[E_{R,{{\cm{\ell}}}}(u,\beta):=\int_{R\times\{{\cm{\ell}}\}}\left[|Du(x)-\beta(x)|^2 \dx+ \eps |\partial_\xi\beta_\eta|+\eps |\partial_\eta\beta_\xi| \right].\]
Since $\partial_\xi u_0=1-\theta$, we have, integrating by parts,
\begin{align*}
 &\int_{R_{{\cm{\ell}}}} (\beta_\xi(x)-(1-\theta)) \varphi_{x_a,a}(x_\xi) f(x_\eta)\dx\\
 =&
\int_{R_{{\cm{\ell}}}} (\partial_\xi u(x)-\partial_\xi u_0(x)) \varphi_{x_a,a}(x_\xi)f(x_\eta)\dx +
\int_{R_{{\cm{\ell}}}} (\beta_\xi(x)-\partial_\xi u(x)) \varphi_{x_a,a}(x_\xi)f(x_\eta)\dx \\
 =&-
\int_{R_{{\cm{\ell}}}}  (u(x)-u_0(x)) \partial_\xi\varphi_{x_a,a}(x_\xi)f(x_\eta)\dx +
\int_{R_{{\cm{\ell}}}} (\beta_\xi(x)-\partial_\xi u(x)) \varphi_{x_a,a}(x_\xi)f(x_\eta)\dx \\
=&-
\int_{R_0}  (u(x)-u_0(x)) \partial_\xi\varphi_{x_a,a}(x_\xi)f(x_\eta) \dx-
\int_{R\times(0,{\cm{\ell}})} (\partial_3 u(x)) (\partial_\xi\varphi_{x_a,a}(x_\xi))f(x_\eta)\dx\\
&+
\int_{R_{{\cm{\ell}}}} (\beta_\xi(x)-\partial_\xi u(x)) \varphi_{x_a,a}(x_\xi)f(x_\eta) \dx\\
\le &
% (\frac{E_R}\mu)^{1/2} \|\partial_\xi\varphi_{x_a,a}f\|_{H^{-1/2}(R)}
-\int_{R_0}  (u(x)-u_0(x)) \partial_\xi\varphi_{x_a,a}(x_\xi)f(x_\eta)\dx
+{\cm ({E_R})^{1/2}   \| \partial_\xi\varphi_{x_a,a}f\|_{L^2(R)}
+E_R^{1/2} \| \varphi_{x_a,a} f\|_{L^2(R)}.}
\end{align*}
 By definition of $\varphi_{x_a,a}$, we have the estimates 
$ \| \varphi_{x_a,a}'f\|_{L^2(R)}\le c a^{-1/2} \|f\|_{L^2((x_b,x_b+b))}$ and
$ \| \varphi_{x_a,a}f\|_{L^2(R)}\le c a^{1/2} \|f\|_{L^2((x_b,x_b+b))}$. 
It remains to estimate the first term of the right-hand side.
Since $|\partial_\xi\varphi_{x_a,a}|\leq c/a$, and since by symmetry of $\varphi_{x_a,a}$ we have $\partial_\xi\varphi_{x_a,a}(x_\xi)=-\partial_\xi\varphi_{x_a,a}(2x_a+a-x_\xi)$ for all $x_\xi\in(x_a,x_a+a)$, we obtain
(writing for brevity $(x_\xi,x_\eta)$ instead of $x_\xi e_\xi+x_\eta e_\eta$)
 \begin{align*}
  \int_{\BZ{R_0}} f(x_\eta) \partial_\xi \varphi_{x_a,a}(x_\xi) &(u-u_0)(x_\xi,x_\eta) \dx\le
  \int_{\BZ{R_0}} |f|  \frac{c}{a}\left| (u-u_0)(x_\xi, x_\eta)-(u-u_0)(2x_a+a-x_\xi, x_\eta)\right| \dx \\
  \le& \frac ca\int_{R_0}  |f|\int_{\{y: |y_3|+|y_\xi-x_a-a/2|=|x_\xi-x_a-a/2|, y_\eta=x_\eta,y_3\le0\}}|D(u-u_0){\BZ{(y)}}| d\calH^1(y) 
  \dx\\
    \le&\frac ca \|f\|_{L^2(R\times(-a,0))}\|D(u-u_0)\|_{L^2(R\times(-a,0))},
     \end{align*}
     which implies that (recall that $f$ depends only on $x_\eta$)
     \begin{eqnarray}
     \label{eq:estnegsob}
       \int_{{\BZ{R_0}}} f(x_\eta) \partial_\xi \varphi_{x_a,a}(x_\xi) (u-u_0)(x_\xi,x_\eta)\dx &\le&\frac{c}{\mu^{1/2}}\|f\|_{L^2((x_b,x_b+b))}E_R^{1/2}.
     \end{eqnarray}
 Putting things together {\cm and using that $a<1$}, {\cm we obtain} the first assertion \eqref{eq:choice1}. The second one follows similarly by performing the same computation in the other direction and using that $\partial_\eta u_0=\theta$.
% \begin{align}
%  \int_{R_\ell} \beta_\eta \varphi_a\varphi_b 
% \ge &
% \theta ab - E_R^{1/2}(\frac{a}{\mu})^{1/2} 
% - {E_R}^{1/2}  \frac{a^{1/2}\ell^{1/2}}{b^{1/2}}
% -{E_R}^{1/2}\frac{a^{1/2}b^{1/2}}{\ell^{1/2}}\,.
% \end{align}
\end{proof}
\begin{corollary}\label{cor:dischoice}
 Under the assumptions of Lemma \ref{lemmatest}, if $E_R(u,\beta)\le c|R|E(u,\beta)$, then
%\begin{equation}\label{eq:cor1}
% E(u,\beta)\gtrsim \min\left\{ a\mu,\frac{a^2}{\ell},\, \ell,\,
% \theta^2b\mu,\,\theta^2\ell,\, \frac{\theta^2b^2}{\ell}
% \right\},
%\end{equation}
at least one of
\begin{equation}\label{eq:cor1}
{\cm  E(u,\beta)\gtrsim \theta^2 \min\left\{ b\mu,\,b^2
 \right\}}
\end{equation}
and
 \begin{equation}
  \int_{R_{{\cm{\ell}}}} \beta_\eta(x) \varphi_{x_b,b}(x_\eta)\dx\ge \frac14\theta |R|
  %\hskip5mm\text{\qquad and\qquad}
 %\int_{R_{{\cm{\ell}}}} \beta_\xi(x)\varphi_{x_a,a}(x_\xi)\dx \ge \frac14|R|\,. 
 \end{equation}
holds.
\end{corollary}
\begin{proof}
Set $g:=1$. Then $\|g\|_{L^2((x_a,x_a+a))}=  a^{1/2}$ and $\int_{R_{{\cm{\ell}}}}g(x_\xi)\varphi_{x_b,b}(x_\xi)\dx>\frac78 ab$.
Thus we have $\int_{R_{{\cm{\ell}}}}\beta_\eta(x)\varphi_{x_b,b}(x_\eta)\dx\geq \frac14 \theta ab$, or 
by \eqref{eq:choice2}
\begin{equation}
 {\cm \theta ab \lesssim E_R^{1/2}(u,\beta) a^{1/2} \left[ \frac{1}{\mu^{1/2}}
 + \frac{1}{b^{1/2}}\right].}
\end{equation}
Rearranging terms and using the assumption $E_R(u,\beta)\le c|R|E(u,\beta)$ yields {\BZ{\eqref{eq:cor1}}}.
\end{proof}
We now use the previous estimates to obtain a lower bound on $E$ similarly to Lemma \ref{lemma1} in the martensite case.
\def\no{Set $f:=1$. Then $\|f\|_{L^2((0,b))}=  b^{1/2}$ and $\int_{R_{{\cm{\ell}}}}f(x_\eta)\varphi_{x_a,a}(x_\xi)\dx>\frac78 ab$.
Thus we have $\int_{R_{{\cm{\ell}}}}\beta_\xi(x)\varphi_{x_a,a}(x_\xi)\dx\geq \frac14 ab$, or 
%\[\frac14 ab>\int_{R_{{\cm{\ell}}}}\beta_\xi\geq\int_{R_{{\cm{\ell}}}}\beta_\xi\varphi_{x_a,a}(\xi)f(\eta), \]
%and in the latter case 
by \eqref{eq:choice1}, using that $1-\theta\geq 1/2$,
\[ab\leq c E_{R}^{1/2}(u,\beta) b^{1/2}\left[
  \frac{1}{\mu^{1/2}}+ (\frac{\ell}{a})^{1/2} 
+(\frac{a}\ell)^{1/2}\right]. \]
Analogously, using \eqref{eq:choice2} with $g:=1$, we have $\int_{R_{{\cm{\ell}}}} \beta_\eta(x) \varphi_{x_b,b}(x_\eta)\dx\ge \frac{1}{4} \theta |R|$, or
\begin{equation}
 \theta ab \le
 c E_R^{1/2}(u,\beta) a^{1/2} \left[ \frac{1}{\mu^{1/2}}
 + (\frac{\ell}b)^{1/2} +(\frac{b}\ell)^{1/2}\right].
\end{equation}
Rearranging terms and using the assumption $E_R(u,\beta)\le c|R|E(u,\beta)$ yields the assertion.}

  \begin{lemma}\label{lem:dislb}
   There are $m_0\in(0,1)$ and $c>0$ with the following property: For all $\eps,\mu>0$, {\BZ{all }}$\theta\in(0,1/2]$ and all $m\in(0,1)$ such that
\begin{equation}\label{lemma1casesdisl}
\text{ at least one of }\hskip1cm \theta\le m\le m_0 \hskip1cm\text{ or }\hskip1cm   m_0=m\le\theta 
\end{equation}
holds, {\cm and for all $\lambda\in (0,1]$}, one has
   \begin{eqnarray}\label{eq:lbdisloc}
   {\cm  \inf E(u,\beta)\geq c\theta^2\min\left\{\mu,\,\frac{\hat{\eps}}{\lambda},\,\lambda^2 m,\,\mu\lambda\ln\frac{1}{m},
 \frac{\lambda m\mu}{\theta}
    \right\}.}
   \end{eqnarray}
\end{lemma}
\begin{proof}
\begin{figure}
\centerline{ \includegraphics[height=5cm]{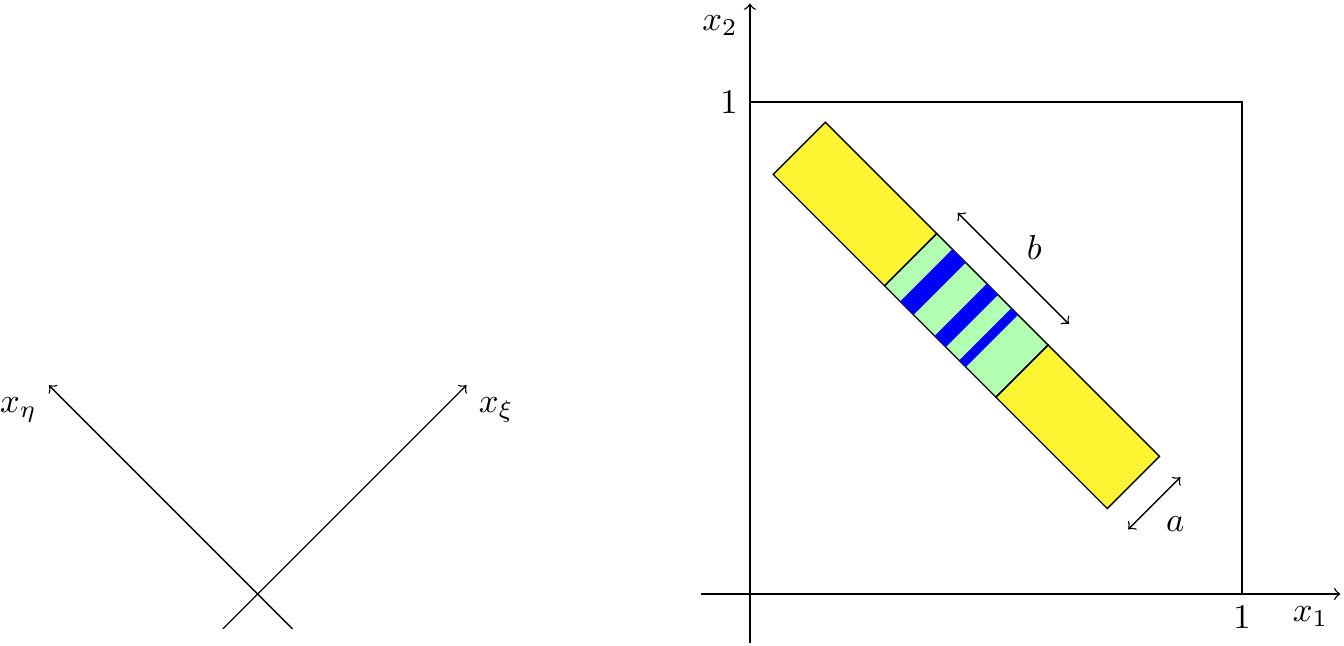}}
 \caption{Left: sketch of the geometry in the proof of Lemma 
 \ref{lem:dislb}. 
 The marked rectangle is $\tilde R$, the inner, light green rectangle is $R$.
 The blue region is $\omega\times(x_a,x_a+a)$ where $\beta_\eta$ concentrates.}
 \label{figlambdislr}
\end{figure}
\emph{Step 1. Preliminaries.}
Let $\eps$, $\mu$, $\theta$, $m$, and $\lambda$ be given with the properties stated in the lemma. It suffices to consider an arbitrary pair $(u,\beta)\in\mathcal{A}$. We use the short-hand notation $E:=E(u,\beta)$, and similarly for {\BZ{$E_R$ and $E_{R,\ell}$}}. We introduce auxiliary parameters $a$, $b$ and $N$. Precisely, fix 
$b\sim 1$, $b\in (0,1/3]$ ($b=1/3$ will do), let $N\in\N$ be such that $\lambda\sim b/N$, and set $a:=\lambda\theta/4$. Choose $x_a$, $x_b\in(0,1)$ such that $\tilde{R}:=(x_a,x_a+a)e_\xi\times (x_b-b,x_b+2b)e_\eta\subset(0,1)^2$ and $E_{\tilde{R}}\lesssim |\tilde{R}|E\lesssim ab E$ (see Figure \ref{figlambdislr}).
Let $R:=(x_a,x_a+a)e_\xi\times(x_b,x_b+b)e_\eta\subset\tilde{R}$. Finally pick ${\cm {\cm{\ell}}\in(0,1)}$ as in Lemma \ref{lemmatest}, i.e., 
such that {\cm $E_{\tilde{R},{{\cm{\ell}}}}\leq E_{\tilde{R}}$}.  \\
We apply Corollary \ref{cor:dischoice} to $R$ which shows that if \eqref{eq:cor1} does not hold then \eqref{eq:disass1} holds with $v:=\beta_\eta\varphi_{x_b,b}$ on $R_{{\cm{\ell}}}$. Since
\[\int_{R_{{\cm{\ell}}}}|\partial_\xi\beta_\eta\varphi_{x_b,b}|\leq\int_{R_{{\cm{\ell}}}}|\partial_\xi\beta_\eta|\leq\frac{1}{\eps}E_{R,{{\cm{\ell}}}}\text{\qquad and\qquad}\int_{R_{{\cm{\ell}}}}|\partial_\eta\beta_\xi|
\leq\frac{1}{\eps}E_{R,{{\cm{\ell}}}}, \]
we deduce from 
 Lemma \ref{lemmarect}, with $w:=\beta_\xi$,  that we have  
\begin{equation}\label{eq:disest1}
 {\cm E\gtrsim \min\left\{ %a\mu,\frac{a^2}{\ell}, \ell,
 \theta^2b\mu,\,\theta^2b^2,\, \frac{\eps \theta}{a}, \frac{\eps N}{b}
 \right\}}
\end{equation}
or there are disjoint intervals $I_1,\dots, I_{8N}\subset (x_b,x_b+b)$ such that \eqref{eq:lemrect} holds with the above choices. 
If (\ref{eq:disest1}) holds, the proof is concluded. 
Suppose now that the other option holds, and set $\omega:=\cup I_i\subset\R$ and $f:=\chi_\omega$.
 Then the second estimate of \eqref{eq:lemrect} yields, since $0\le\varphi_{x_a,a}\le 1$,
\begin{equation*}
\int_{R_{{\cm{\ell}}}} \beta_\xi(x) f(x_\eta)\varphi_{x_a,a}(x_\xi) \dx\le \frac14 a |\omega|\,.
\end{equation*}
By Lemma \ref{lemmatest}, since $(1-\theta)\int_{R_{{\cm \ell}}} f(x_\eta) 
\varphi_{x_a,a}(x_\xi){\BZ{\dx}} \ge \frac12 |\omega|\frac78a$,
\begin{align*}
 a|\omega|\lesssim
{\cm |\omega|^{1/2} (E_R)^{1/2} \left[\frac{1}{\mu^{1/2}}+\frac{1}{a^{1/2}}\right].}
\end{align*}
Therefore, if $|\omega|> mb$ then 
\begin{equation}\label{eq:disest2}
 {\cm E\gtrsim\min\{ am\mu,\, a^2m \}.}
\end{equation}
If (\ref{eq:disest2}) holds, then the proof is concluded 
(in proving $\theta^2 \ell \lesssim m\ell$
we use that (\ref{lemma1casesdisl}) implies $\theta\lesssim m$).
If instead $|\omega|\le mb$, 
we proceed along the lines of the proof of Lemma \ref{lemma1}.
%, separating the two cases 
%in (\ref{lemma1casesdisl}).

%\emph{Step 2. The case $\theta\le m_0$.}
\emph{Step 2. Construction of the test function and conclusion of the proof.}
Let $\hat I_i := B_{mb/N}(I_i)=(y_i-g_i,y_i+g_i)$, and set for $i=1,\dots,8N$,
 \begin{equation}
  \psi(x_\eta):=\max_i \psi_i(x_\eta-y_i)\,,\hskip1cm
  \psi_i(t):=\left[ \ln \frac 1{6m} - \left(\ln \frac{|t|}{g_i}\right)_+ \right]_+\,.
\end{equation}
Note that $\psi$ has compact support in $(x_b-b,x_b+2b)$,  $\psi=\ln\frac{1}{6m}$ in $\omega$,
$g_i\ge {\BZ{(mb)/N}}$, and $\sum_i g_i\le mb+|\omega|\le 2 mb$. Then by the first of \eqref{eq:lemrect} and 
%  obeys
%  \begin{align}
%   \|\psi\|_1&\le cb\\
%   \|\psi\|_2&\le cb^{1/2}\\
%  \|\psi'\|_2&\le c\frac{N}{(mb)^{1/2}}\\
% % \|\psi'\varphi_a\|_{H^{-1/2}(R)}&\le ....
%  \|\psi'\|_{H^{-1/2}(0,b)}^2&\le cN \ln\frac1m\,.
%  \end{align}
% \end{lemma}
% \begin{proof} By assumption $\sum_i g_i\le 2mb$.
% One sees that $\psi_i=\ln\frac1m$ in $(-g_i,g_i)$, $\psi_i=0$ outside $(-g_i/m,
% g_i/m)$. Further, $\|\psi_i\|_{L^1}\le c g_i/m$ and therefore 
% $  \|\psi\|_{L^1} \le cb$. 
% Analogously $  \|\psi\|_{L^2}  \le cb^{1/2}$. 
% We compute $\|\psi'\|_{L^2}^2 \le \sum_i 1/g_i \le cN/(mb/N) \le cN^2/(mb)$. 
% 
% {\bf need to do something on the boundary values?}
% We estimate the $H^{1/2}$ norm of $\psi_i$ by the $H^1$ norm of its
% radially symmetric extension to $\R^2$,
% \begin{equation}
%   \|\psi_i\|_{H^{1/2}(\R)}^2\le
%   \|D\psi_i\|_{L^2(\R^2)}^2=\int_{g_i}^{g_i/m} 2\pi r \frac{1}{r^2} dr
%    = 2\pi \ln \frac1m\,.
% \end{equation}
% Analogously,
% \begin{equation}
%   \|\psi\|_{H^{1/2}(\R)}^2\le
%   \|D\psi\|_{L^2(\R^2)}^2\le \sum_i   \|D\psi_i\|_{L^2(\R^2)}^2\le  cN \ln \frac1m\,. 
% \end{equation}
% {\bf In one dimension, $\|\psi'\|_{-1/2}=\|\psi\|_{1/2}$??}
% \end{proof}
% Finally, we test $\beta_\eta$ with $\psi(\eta)\varphi_a(\xi)$.
% {\bf Assume $\beta_\eta\ge0$}.
since $\beta_\eta$, $\psi$ and $\varphi_1$ are non-negative, we have, using $\partial_\eta u_0=\theta$,
\begin{align*}
&\theta|R| \ln \frac{1}{\BZ{6}m} \lesssim
 \int_{\tilde{R}_{{\cm{\ell}}}} \beta_\eta(x) \varphi_{x_b,b}(x_\eta)\psi(x_\eta)\dx
\le
 \int_{\tilde{R}_{{\cm{\ell}}}} \beta_\eta(x) \psi(x_\eta)\dx
 \\
=&\int_{\tilde{R}_{{\cm{\ell}} }}\psi(x_\eta)(\beta_\eta(x)-\partial_\eta u(x))\dx-\int_{\tilde{R}_0}\psi'(x_\eta)(u(x)-u_0(x))\dx-\int_{\tilde{R}\times(0,{\cm{\ell}})}\partial_3 u(x)\psi'(x_\eta)\dx+ \int_{\tilde{R}}\psi(x_\eta)\theta\dx\\
\lesssim&{\cm E_R^{1/2} \|\psi\|_{L^2(\tilde R)}+\int_{\tilde{R}_0}(u_0(x)-u(x))\psi'(x_\eta)\dx+E_R^{1/2}
\|\psi'\|_{L^2(\tilde{R}\times(0,1))} +  \theta \|\psi\|_{L^1(\tilde R)} \,.}
\end{align*}
By Lemma \ref{lemmah12} and Lemma \ref{lem:H1/2average}, we have
\begin{align*}
\int_{\tilde{R}_0}(u_0(x)-u(x))\psi'(x_\eta)\dx\lesssim&[\psi]_{H^{1/2}_N((x_b-b,x_b+2b))}\left[\int_{x_a}^{x_a+a}(u_0-u)(x_\xi,\cdot)\dx_\xi\right]_{H^{1/2}_N((x_b-b,x_b+2b))}\\
\lesssim& a^{1/2}[\psi]_{H^{1/2}_N((x_b-b,x_b+2b))}[u_0-u]_{H^{1/2}_N(\tilde{R})}.
\end{align*}
%We estimate the $H^{-1/2}$ norm of $\varphi_a \psi'$ with
%\begin{equation}
% \int_{R} \psi'\varphi_a H = a \int_{(0,b)}
% \psi' h
%\end{equation}
%where $h=\frac1a \int_{(0,a)}\varphi_a H$. Clearly 
%$\int_{(0,b)\times(0,\infty)} |Dh|^2\le 
%a^{-1}\int_{R\times(0,\infty)} |DH|^2$, which gives $\|h\|_{H^{1/2}}\le a^{-1/2} \|H\|_{H^{1/2}}$.
%Therefore
%\begin{equation}
% \int_{R} \psi'\varphi_a H = a \int_{(0,b)}
% \psi' h \le a \|\psi'\|_{H^{-1/2}} a^{-1/2} \|H\|_{H^{1/2}}
% \le a^{1/2} ( N \ln\frac1m)^{1/2}\,.
%\end{equation}
We use that (cf. Step 2 in the proof of Lemma \ref{lemma1}) $\|\psi\|_{L^2}\lesssim b^{1/2}$, $[\psi]_{H^{1/2}_N}\lesssim N^{1/2}\ln^{1/2}\frac{1}{{\BZ{6}}m}$, $\|\psi'\|_{L^2}\lesssim N/(mb)^{1/2}$ and $\|\psi\|_{L^1}\lesssim b$, 
where the norms are taken on $(x_b-b,x_b+2b)$. Hence,
\[\theta ab\ln\frac{1}{{\BZ{6}}m}\leq c E_R^{1/2}{\cm \left(\left(ab\right)^{1/2}+
\left(\frac{aN\ln\frac{1}{{\BZ{6}}m}}\mu\right)^{1/2}+N\left(\frac{a}{mb}\right)^{1/2}\right)}+c\theta ab. \]
If $m_0$ is small enough, then the last term can be absorbed into the left-hand side {\BZ{and $\ln\frac{1}{6m}\gtrsim\ln\frac{1}{m}$}}. Hence,
\begin{equation}\label{eq:disest3}
{\cm E\gtrsim \theta^2\min\big\{\ln^2\frac{1}{m},\,\frac{\mu b\ln\frac{1}{m}}{N}, \frac{mb^2\ln^2\frac{1}{m}}{ N^2}\big\}=\theta^2\min\big\{\frac{\mu b\ln\frac{1}{m}}{N}, \frac{mb^2\ln^2\frac{1}{m}}{ N^2}\big\}.}
\end{equation}
Recalling $b\sim 1$, $N\sim 1/\lambda$, $m/\theta\gtrsim 1$, and $\ln 1/m\gtrsim 1$ the proof is concluded.
% \begin{equation}
%  E\ge c \min\left\{ a\mu,\frac{a^2}{\ell}, \ell,
%  \theta^2b\mu,\frac{\theta^2b^2}{\ell}, \theta^2\ell, \frac{\ell\eps \theta}a, \frac{\ell\eps N}{b}
%  \right\}
% \end{equation}
% or
% \begin{equation}
%  E\ge c\min\{ am\mu, \frac{a^2m}{\ell},   m\ell \}\,.
% \end{equation}
% or
% \begin{equation}
%  E\ge c\min\{ \theta^2 \frac{\mu b \ln\frac1m}{N}, \frac{mb^2\theta^2\ln^2\frac1m}{\ell N^2}, \theta^2 \ell \ln^2\frac1m
%  \}
% \end{equation}
% Insert $b=\lambda N$, drop irrelevant terms.
% \begin{equation}
%  E\ge c \min\left\{\cdot,\cdot, \cdot,
%  \theta^2\lambda N\mu,\frac{\theta^2\lambda^2N^2}{\ell}, \theta^2\ell, \frac{\ell\eps \theta}a, \frac{\ell\eps }{\lambda}
%  \right\}
% \end{equation}
% or
% \begin{equation}
%  E\ge c\min\{ am\mu, \frac{a^2m}{\ell},  \cdot \}\,.
% \end{equation}
% or
% \begin{equation}
%  E\ge c\min\{ \theta^2 \mu \lambda \ln\frac1m, \frac{m\lambda^2\theta^2\ln^2\frac1m}{\ell}, \cdot
%  \}
% \end{equation}
% 
%
% 
% {\bf Parameter identification:}
% Choose 
% \begin{equation}
% 
% \end{equation}
% Then 
%\begin{eqnarray*}
% \inf E[u,\beta]\geq c\min\{\mu\theta^2,\,\theta^2/\ell,\,\theta^2\ell,\,\ell\eps/\lambda,\,\cdot,\,\theta\lambda m\mu,\,\theta^2\lambda^2 m/\ell,\,\theta^2\mu\lambda\ln\frac{1}{m}\,\cdot\}.
%\end{eqnarray*}

%\emph{Step 3. The case $\theta\ge m_0$.}
\end{proof}
Proceeding along the lines of the proof of Proposition \ref{prop:KMlb}, 
we now conclude the proof of the lower bound.
\begin{proposition}\label{prop:disloclb}
 There is a constant $c>0$ such that for all $\eps,\mu>0$, and all $0<\theta\le1/2$, we have
 \[\inf E(u,\beta)\ge c\theta^2\min\left\{\hat{\eps}^{2/3},\,(\hat{\eps} \mu)^{1/2}\ln^{1/2}\left(3+\frac{\hat{\eps}}{\mu^3 }\right),\,\left(\hat{\eps} \mu\right)^{1/2}\ln^{1/2}\left(\frac{1}{\theta}\right),\, \mu,\ 1\right\}.\]
\end{proposition}
\begin{proof}
We start fixing $(u,\beta)\in \mathcal A$, $\eps$, $\mu>0$ and $\theta\in(0,1/2]$, and 
use Lemma \ref{lem:dislb} with different choices of parameters in different regimes.
\begin{itemize}
\item[(i)] If $\hat{\eps}\leq\min\{1,\ \mu^3\}$, we choose $ m:=m_0$ and $\lambda:=\hat{\eps}^{1/3}$. 
Then 
%(\ref{lemma1casesdisl}) holds for all $\theta$, and 
we obtain $E(u,\beta)\gtrsim\theta^2\min\{\hat{\eps}^{2/3},\ \mu\}$. 
\item[(ii)] If $1\leq\hat{\eps}\leq\mu^3 $, we choose $ m:=m_0$ {\cm and $\lambda:=1$.} 
Then by \eqref{eq:lbdisloc},
$E(u,\beta)\geq c\theta^2\min\{\mu, \ 1\}$. 
\item[(iii)] It remains to consider $\hat{\eps}>\mu^3$.
We set
\begin{equation}\label{eqdefm2}
m:=\max\{\min\{m_0,\theta \ln \frac2\theta\}, \left(\frac{\mu^{3}}{\hat{\eps}}\right)^{1/4}m_0\}\,.
\end{equation}
Since $\theta\le 1/2$ we have $\ln\frac2\theta\ge 1$, therefore (\ref{lemma1casesdisl}) holds. We distinguish two subcases:
\begin{itemize}
\item[(a)] If $\hat{\eps}\le {\mu}\ln\frac{1}{ m}$, we choose $\lambda:=\left(\frac{\hat{\eps}}{\mu\ln\frac{1}{ m}}\right)^{1/2}$. 
Using $m\ge m^2(\ln \frac1m)^{3/2}$ in the {\cm $\lambda^2 m$} term, 
{\SC
and $m\gtrsim \theta \ln \frac1\theta\gtrsim \theta \ln \frac1m$ in the $\lambda m\mu/\theta$ term, }we conclude
$E(u,\beta)\gtrsim \theta^2\min\{ \mu,1,(\hat{\eps}\mu)^{1/2}\ln^{1/2}\frac1m\}$.
Since $\ln 1/(\theta \ln 2/\theta)\sim \ln 1/\theta$, the proof is concluded in this case.
\item[(b)] If $\hat{\eps}>\mu\ln\frac{1}{ m}$, we choose $\lambda:=1$ and $m:=m_0$. 
Then $E(u,\beta)\gtrsim\theta^2\min\{1,\mu\}$.
\end{itemize}
\end{itemize}
\end{proof}

\section*{Acknowledgements}
We thank {\BZ{Michael Goldman}} and Stefan M\"uller for several interesting discussions.
This work was partially supported by the Deutsche Forschungsgemeinschaft
through the  Sonderforschungsbereich 1060 {\em ``The mathematics of emergent effects''}, 
projects A5 and  A6. 

% \bibliographystyle{plain}
%\bibliographystyle{alpha-noname-nonumber}
%\bibliography{kleinbeta}

\begin{thebibliography}{BCLdMQ12}

\bibitem[ACO09]{alberti-et-al:09}
G.~Alberti, R.~Choksi, and F.~Otto.
\newblock Uniform energy distribution for an isoperimetric problem with
  long-range interactions.
\newblock {\em J. Amer. Math. Soc.}, 22:569--605, 2009.

\bibitem[AD14]{AnD14}
K.~Anguige and P.~Dondl.
\newblock Relaxation of the single-slip condition in strain-gradient
  plasticity.
\newblock {\em R. Soc. Lond. Proc. Ser. A Math. Phys. Eng. Sci.}, 470, 2014.

\bibitem[AD15]{DondlMicroplast}
K.~Anguige and P.~Dondl.
\newblock Energy estimates, relaxation, and existence for strain-gradient
  plasticity with cross-hardening.
\newblock In S.~Conti and K.~Hackl, editors, {\em Analysis and computation of
  microstructure in finite plasticity}, volume~78, pages 157--173. Springer,
  2015.

\bibitem[BBCDM02]{belgacem-et-al:02}
H.~Ben~Belgacem, S.~Conti, A.~DeSimone, and S.~M\"uller.
\newblock Energy scaling of compressed elastic films.
\newblock {\em Arch. Rat. Mech. Anal.}, 164:1--37, 2002.

\bibitem[BCDM00]{BCDM00}
H.~{Ben Belgacem}, S.~Conti, A.~DeSimone, and S.~M\"uller.
\newblock Rigorous bounds for the {F\"oppl-von K\'arm\'an} theory of
  isotropically compressed plates.
\newblock {\em J. Nonlinear Sci.}, 10:661--683, 2000.

\bibitem[BCLdMQ12]{bechthold-et-el:12}
C.~Bechtold, C.~Chluba, R.~Lima~de Miranda, and E.~Quandt.
\newblock High cyclic stability of the elastocaloric effect in sputtered tinicu
  shape memory films.
\newblock {\em Applied Physics Letters}, 101, 2012.

\bibitem[BG]{bella-goldman:15}
P.~Bella and M.~Goldman.
\newblock Nucleation barriers at corners for cubic-to-tetragonal phase
  transformation.
\newblock to appear in Proc. Roy. Soc. Edinburgh A.

\bibitem[BJ87]{ball-james:87}
J.~Ball and R.~James.
\newblock Fine phase mixtures as minimizers of energy.
\newblock {\em Arch. Rat. Mech. Anal.}, 100:13--52, 1987.

\bibitem[BK14]{bella-kohn:14}
P.~Bella and R.~V. Kohn.
\newblock Wrinkles as the result of compressive stresses in an annular thin
  film.
\newblock {\em Communications on Pure and Applied Mathematics}, 67:693--747,
  2014.

\bibitem[CC14]{chan-conti:14}
A.~Chan and S.~Conti.
\newblock Energy scaling and domain branching in solid-solid phase transitions.
\newblock In M.~Griebel, editor, {\em Singular Phenomena and Scaling in
  Mathematical Models}, pages 243--260. Springer International Publishing,
  2014.

\bibitem[CC15]{chan-conti:14-1}
A.~Chan and S.~Conti.
\newblock Energy scaling and branched microstructures in a model for
  shape-memory alloys with {SO}(2) invariance.
\newblock {\em Math. Models Methods App. Sci.}, 25:1091--1124, 2015.

\bibitem[CCF{\etalchar{+}}06]{cui-et-al:06}
J.~Cui, Y.~Chu, O.~Famodu, Y.~Furuya, J.~Hattrick-Simpers, R.~James, A.~Ludwig,
  S.~Thienhaus, M.~Wuttig, Z.~Zhang, and I.~Takeuchi.
\newblock Combinatorial search of thermoelastic shape-memory alloys with
  extremely small hysteresis width.
\newblock {\em Nature materials}, 5:286--290, 2006.

\bibitem[CCKO08]{choksi-et-al:08}
R.~Choksi, S.~Conti, R.~Kohn, and F.~Otto.
\newblock Ground state energy scaling laws during the onset and destruction of
  the intermediate state in a type \mbox{I} superconductor.
\newblock {\em Comm. Pure Appl. Math.}, 61(5):595--626, 2008.

\bibitem[CGM11]{ContiGarroniMueller2011}
S.~Conti, A.~Garroni, and S.~M\"uller.
\newblock Singular kernels, multiscale decomposition of microstructure, and
  dislocation models.
\newblock {\em Arch. Rat. Mech. Anal.}, 199:779--819, 2011.

\bibitem[CGO15]{ContiGarroniOrtiz}
S.~Conti, A.~Garroni, and M.~Ortiz.
\newblock The line-tension approximation as the dilute limit of linear-elastic
  dislocations.
\newblock {\em Arch. Rat. Mech. Anal.}, 2015.
\newblock to appear.

\bibitem[Cho01]{Choksi01}
R.~Choksi.
\newblock Scaling laws in microphase separation of diblock copolymers.
\newblock {\em J. Nonlinear Sci.}, 11:223--236, 2001.

\bibitem[CK98]{choksi-kohn:98}
R.~Choksi and R.~V. Kohn.
\newblock Bounds on the micromagnetic energy of a uniaxial ferromagnet.
\newblock {\em Comm. Pure Appl. Math.}, 51:259--289, 1998.

\bibitem[CKO98]{choksi-et-al:98}
R.~Choksi, R.~Kohn, and F.~Otto.
\newblock Domain branching in uniaxial ferromagnets: a scaling law for the
  minimum energy.
\newblock {\em Comm. Math. Phys.}, 201(1):61--79, 1998.

\bibitem[CO05]{conti-ortiz:05}
S.~Conti and M.~Ortiz.
\newblock Dislocation microstructures and the effective behavior of single
  crystals.
\newblock {\em Arch. Rat. Mech. Anal.}, 176:103--147, 2005.

\bibitem[CO09]{capella-otto:09}
A.~Capella and F.~Otto.
\newblock A rigidity result for a perturbation of the geometrically linear
  three-well problem.
\newblock {\em Comm. Pure Appl. Math.}, 62:1632--1669, 2009.

\bibitem[CO12]{CapellaOtto2012}
A.~Capella and F.~Otto.
\newblock A quantitative rigidity result for the cubic-to-tetragonal phase
  transition in the geometrically linear theory with interfacial energy.
\newblock {\em Proc. Roy. Soc. Edinburgh Sect. A}, 142:273--327, 2012.

\bibitem[Con00]{conti:00}
S.~Conti.
\newblock Branched microstructures: scaling and asymptotic self-similarity.
\newblock {\em Comm. Pure Appl. Math.}, 53:1448--1474, 2000.

\bibitem[Con06]{conti:06}
S.~Conti.
\newblock A lower bound for a variational model for pattern formation in
  shape-memory alloys.
\newblock {\em Cont. Mech. Thermodyn.}, 17 (6):469--476, 2006.

\bibitem[COS15]{conti-et-al:15}
S.~Conti, F.~Otto, and S.~Serfaty.
\newblock Branched microstructures in the \text{G}inzburg-\text{L}andau model
  of type-\text{I}-superconductors.
\newblock arxiv:1507.00836v1, 2015.

\bibitem[Dac07]{Dacorognabuch}
B.~Dacorogna.
\newblock {\em Direct methods in the calculus of variations}, volume~78.
\newblock Springer, 2007.

\bibitem[Die10]{diermeier:10}
J.~Diermeier.
\newblock Nichtkonvexe {V}ariationsprobleme und {M}ikrostrukturen.
\newblock Bachelor's thesis, Universit\"at Bonn, 2010.

\bibitem[Die13]{diermeier:13}
J.~Diermeier.
\newblock Domain branching in linear elasticity.
\newblock Master's thesis, Universit\"at Bonn, 2013.

\bibitem[DKZ{\etalchar{+}}10]{delville-et-al}
R.~Delville, S.~Kasinathan, Z.~Zhang, J.~Van~Humbeeck, R.~James, and
  D.~Schryvers.
\newblock Transmission electron microscopy study of phase compatibility in low
  hysteresis shape memory alloys.
\newblock {\em Philosophical Magazine}, 90:177–195, 2010.

\bibitem[DSZJ09]{delville-et-al2}
R.~Delville, D.~Schryvers, Z.~Zhang, and R.~James.
\newblock Transmission electron microscopy investigation of microstructures in
  low-hysteresis alloys with special lattice parameters.
\newblock {\em Scripta Materialia}, 60:293--296, 2009.

\bibitem[GLP10]{GarroniLeoniPonsiglione2010}
A.~Garroni, G.~Leoni, and M.~Ponsiglione.
\newblock Gradient theory for plasticity via homogenization of discrete
  dislocations.
\newblock {\em J. Eur. Math. Soc. (JEMS)}, 12:1231--1266, 2010.

\bibitem[GM06]{GarroniMueller2006}
A.~Garroni and S.~M{\"u}ller.
\newblock A variational model for dislocations in the line tension limit.
\newblock {\em Arch. Ration. Mech. Anal.}, 181:535--578, 2006.

\bibitem[GM12]{giuliani-mueller:12}
A.~Giuliani and S.~M\"uller.
\newblock Striped periodic minimizers of a two-dimensional model for
  martensitic phase transitions.
\newblock {\em Comm. Math. Phys.}, 309:313--339, 2012.

\bibitem[JS01]{JinSternberg2}
W.~Jin and P.~Sternberg.
\newblock Energy estimates of the {von K\'arm\'an} model of thin-film
  blistering.
\newblock {\em J. Math. Phys.}, 42:192--199, 2001.

\bibitem[JS02]{JinSternberg1}
W.~Jin and P.~Sternberg.
\newblock In-plane displacements in thin-film blistering.
\newblock {\em Proc. R. Soc. Edin. A}, 132A:911--930, 2002.

\bibitem[JZ05]{james-zhang:05}
R.~James and Z.~Zhang.
\newblock A way to search for multiferroic materials with \"{}unlikely\"{}
  combinations of physical properties.
\newblock In L.~Manosa, A.~Planes, and A.~Saxena, editors, {\em The Interplay
  of Magnetism and Structure in Functional Materials}, volume~79. Springer,
  2005.

\bibitem[KK11]{knuepfer-kohn:11}
H.~Kn{\"u}pfer and R.~V. Kohn.
\newblock Minimal energy for elastic inclusions.
\newblock {\em Proc. R. Soc. Lond. Ser. A Math. Phys. Eng. Sci.}, 467:695--717,
  2011.

\bibitem[KKO13]{knuepfer-kohn-otto:13}
H.~Kn\"upfer, R.~V. Kohn, and F.~Otto.
\newblock Nucleation barriers for the cubic-to-tetragonal phase transformation.
\newblock {\em Comm. Pure Appl. Math.}, 66:867--904, 2013.

\bibitem[KM92]{kohn-mueller:92}
R.~Kohn and S.~M\"{u}ller.
\newblock Branching of twins near an austenite-twinned martensite interface.
\newblock {\em Phil. Mag. A}, 66:697--715, 1992.

\bibitem[KM94]{kohn-mueller:94}
R.~Kohn and S.~M\"uller.
\newblock Surface energy and microstructure in coherent phase transitions.
\newblock {\em Comm. Pure Appl. Math.}, XLVII:405--435, 1994.

\bibitem[KM11]{knuepfer-muratov:11}
H.~Kn\"upfer and C.~Muratov.
\newblock Domain structure of bulk ferromagnetic crystals in applied fields
  near saturation.
\newblock {\em J. Nonlinear Sc.}, pages 1--42, 2011.

\bibitem[KW]{Kohn-Wirth:15}
R.~Kohn and B.~Wirth.
\newblock Optimal fine-scale structures in compliance minimization for a shear
  load.
\newblock {\em Communications in Pure and Applied Mathematics}.
\newblock accepted.

\bibitem[KW14]{Kohn-Wirth:14-1}
R.~Kohn and B.~Wirth.
\newblock Optimal fine-scale structures in compliance minimization for a
  uniaxial load.
\newblock {\em Proceedings of the Royal Society A}, 470:20140432--20140448,
  2014.

\bibitem[Lan38]{landau:38}
L.~Landau.
\newblock The intermediate state of supraconductors.
\newblock {\em Nature}, 141:688, 1938.

\bibitem[Lan43]{landau:43}
L.~Landau.
\newblock On the theory of the intermediate state of superconductors.
\newblock {\em J. Phys. USSR}, 7:99, 1943.

\bibitem[Leo09]{leoni}
G.~Leoni.
\newblock {\em A first course in Sobolev Spaces}.
\newblock Graduate Studies in Mathematics. AMS, 2009.

\bibitem[LKBH10]{louie-et-al}
M.~Louie, M.~Kislitsyn, K.~Bhattacharya, and S.~Haile.
\newblock Phase transformation and hysteresis behavior in
  \mbox{Cs}$_{1-x}$\mbox{Rb}$_{x}$\mbox{H}$_2$\mbox{PO}$_4$.
\newblock {\em Solid state ionics}, 181:173--179, 2010.

\bibitem[MSZ14]{MuellerScardiaZeppieri2014}
S.~M\"uller, L.~Scardia, and C.~I. Zeppieri.
\newblock Geometric rigidity for incompatible fields and an application to
  strain-gradient plasticity.
\newblock {\em Indiana Univ. Math. J.}, 63:1365--1396, 2014.

\bibitem[M{\"u}l99]{MuellerLectureNotes}
S.~M{\"u}ller.
\newblock Variational models for microstructure and phase transitions.
\newblock In F.~Bethuel et~al., editors, {\em {Calculus} of variations and
  geometric evolution problems}, Springer Lecture Notes in Math. 1713, pages
  85--210. Springer-Verlag, 1999.

\bibitem[OV10]{otto-viehmann:10}
F.~Otto and T.~Viehmann.
\newblock Domain branching in uniaxial ferromagnets: asymptotic behavior of the
  energy.
\newblock {\em Calc. Var. PDE}, 38:135--181, 2010.

\bibitem[Sch94]{schreiber:94}
C.~Schreiber.
\newblock Rapport de stage, d.e.a.
\newblock Freiburg, 1994.

\bibitem[SCJ10]{Sriva-et-al}
V.~Srivastava, X.~Chen, and R.~James.
\newblock Hysteresis and unusual magnetic properties in the singular
  \mbox{H}eusler alloy
  $\mbox{Ni}_{45}\mbox{Co}_{5}\mbox{Mn}_{40}\mbox{Sn}_{10}$.
\newblock {\em Appl. Phys. Lett.}, 97:014101, 2010.

\bibitem[SDS{\etalchar{+}}14]{Shi-et-al:14}
H.~Shi, R.~Delville, V.~Srivastava, R.~James, and D.~Schryvers.
\newblock Microstructural dependence on middle eigenvalue in ti–ni–au.
\newblock {\em Journal of Alloys and Compounds}, 582:703 -- 707, 2014.

\bibitem[Zha07]{zhang:07}
Z.~Zhang.
\newblock {\em Special lattice parameters and the design of low hysteresis
  materials}.
\newblock PhD thesis, University of Minnesota, 2007.

\bibitem[ZTY{\etalchar{+}}10]{zarnetta-et-al}
R.~Zarnetta, R.~Takahashi, M.~Young, A.~Savan, Y.~Furuya, S.~Thienhaus,
  B.~Maa\ss, M.~Rahim, J.~Frenzel, H.~Brunken, Y.~Chu, V.~Srivastava, R.~James,
  I.~Takeuchi, G.~Eggeler, and A.~Ludwig.
\newblock Identification of quaternary shape memory alloys with near-zero
  thermal hysteresis and unprecedented functional stability.
\newblock {\em Advanced Functional Materials}, 20(12):1917--1923, 2010.

\bibitem[Zwi14]{zwicknagl:14}
B.~Zwicknagl.
\newblock Microstructures in low-hysteresis shape memory alloys: Scaling
  regimes and optimal needle shapes.
\newblock {\em Arch. Rat. Mech. Anal.}, 213:355--421, 2014.

\end{thebibliography}
\newcommand{\etalchar}[1]{$^{#1}$}

\end{document}